\documentclass{amsart}

\usepackage{booktabs}
\usepackage{longtable}

\usepackage{dynkin-diagrams}

\usepackage{amsmath,amssymb,amsthm}
\usepackage{enumitem}
\usepackage{mathtools}
\usepackage[all,cmtip]{xy}

\newtheorem{theorem}{Theorem}[section]
\newtheorem{lemma}[theorem]{Lemma}
\newtheorem{corollary}[theorem]{Corollary}
\newtheorem{proposition}[theorem]{Proposition}

\newtheorem{claim}[theorem]{Claim}

\theoremstyle{definition}

\newtheorem{remark}[theorem]{Remark}

\newtheorem{example}[theorem]{Example}

\newtheorem{observation}[theorem]{Observation}
\newtheorem*{acknowledgements}{Acknowledgements}

\numberwithin{equation}{section}

\DeclareMathOperator{\Ker}{Ker}
\DeclareMathOperator{\Pic}{Pic}

\DeclareMathOperator{\pr}{pr}
\DeclareMathOperator{\Gr}{Gr}

\DeclareMathOperator{\disc}{disc}
\DeclareMathOperator{\SU}{SU}

\DeclareMathOperator{\td}{td}
\DeclareMathOperator{\ch}{ch}
\DeclareMathOperator{\Res}{Res}
\DeclareMathOperator{\Ev}{Ev}

\DeclareMathOperator{\Ext}{Ext}
\DeclareMathOperator{\Nef}{Nef}
\DeclareMathOperator{\Psef}{Psef}
\DeclareMathOperator{\SO}{SO}
\DeclareMathOperator{\Ad}{Ad}
\DeclareMathOperator{\Spin}{Spin}

\newcommand{\trans}[1]{{}^t\!{#1}}

\newcommand{\sO}{\mathcal{O}}
\newcommand{\sE}{\mathcal{E}}
\newcommand{\sF}{\mathcal{F}}

\newcommand{\sM}{\mathcal{M}}

\newcommand{\sS}{\mathcal{S}}
\newcommand{\sN}{\mathcal{N}}

\newcommand{\sQ}{\mathcal{Q}}
\newcommand{\sK}{\mathcal{K}}

\newcommand{\sU}{\mathcal{U}}
\newcommand{\sP}{\mathcal{P}}
\newcommand{\sH}{\mathcal{H}}
\newcommand{\sW}{\mathcal{W}}

\newcommand{\bC}{\mathbf{C}}

\newcommand{\bN}{\mathbf{N}}
\newcommand{\bP}{\mathbf{P}}
\newcommand{\bQ}{\mathbf{Q}}

\newcommand{\bZ}{\mathbf{Z}}

\newcommand{\blank}{{-}}

\AtBeginDocument{%
   \def\MR#1{}
}

\setenumerate{
label=(\arabic*),
font=\normalfont
}

\begin{document}

\title{Polarized K3 surfaces of genus thirteen and curves of genus three}

\author{Akihiro Kanemitsu}
\date{\today}
\address{Department of Mathematics, Graduate School of Science and Engineering, Saitama University, Saitama-City, Saitama, 338-8570, Japan}
\email{kanemitsu@mail.saitama-u.ac.jp}
\thanks{Supported by JSPS KAKENHI Grant Numbers 21K20316 and 23K12948.}

\author{Shigeru Mukai}
\date{\today}
\address{Research Institute for Mathematical Sciences, Kyoto University, Kyoto 606-8502, Japan}
\email{mukai@kurims.kyoto-u.ac.jp}
\thanks{Supported by JSPS KAKENHI Grant Numbers 16H06335, 20H00112, 18K03231 and by the Research Institute for Mathematical Sciences, an International Joint Usage/Research Center located in Kyoto University.}



\begin{abstract}
We describe a general (primitively) polarized K3 surface $(S,h)$ with $(h^2)=24$ as a complete intersection variety with respect to vector bundles on the $6$-dimensional moduli space $\sN^-$ of the stable vector bundles of rank two with fixed odd determinant on a curve $C$ of genus $3$.
If the curve $C$ is hyperelliptic, then $\sN^-$ is a subvariety of the $12$-dimensional Grassmann variety $\Gr(\bC^8,2)$ defined by a pencil of quadric forms.
In this case, our description implies that a general  $(S,h)$ is the intersection of two (7-dimensional) contact homogeneous varieties of $\Spin(7)$ in the Grassmann variety $\Gr(\bC^8,2)$.
\end{abstract}

\maketitle

\section{Introduction}

The moduli space  $\sF_g$ of (primitively) polarized K3 surfaces of genus $g$ is a $19$-dimensional quasi-projective variety.
In this article we describe the generic member $(S, h) \in \sF_{13}$ of genus $13$ in the moduli space of vector bundles on a curve of genus $3$, by a method which may work in another genus
 (see Section~\ref{sect:g_19}).
 We also give a description in the $12$-dimensional Grassmann variety $\Gr(\bC^8,2)$ (Corollary~\ref{corollary:general_K3_in_G(8,2)}).

There are two methods of description of K3 surfaces $S$.
One is the Grassmannian method using $\Phi_{|E|}: S  \longrightarrow \Gr(H^0(E), r(E))$.
We embed the generic  $(S, h) \in \sF_g$  into the Grassmann variety by the complete linear system $|E|$ attached with a vector bundle  $E$  on  $S$, and describe the image.
For example, Mukai \cite{Muk06} used a rank 3 vector bundle with $h^0(E) = 7$ and described  generic $S$  of genus $13$ in another $12$-dimensional Grassmann variety  $\Gr(\bC^7, 3)$ as the zero locus of a homogeneous vector bundle of rank $10$ (Remark~\ref{rem:semi_rigid_bundles}).

The other, which we employ here, is the VBAC (vector bundles on algebraic curves) method, which has two sides: \emph{surface} and \emph{curve}.
From surface side, it utilizes a family  $\{E_c \coloneqq \sE|_{S\times c}\}_{c \in C}$ of stable vector bundles on  $S$  parametrized by a curve  $C$
(and from curve side vice versa),
where $\sE$  is a vector bundle on the product  $S \times C$.
The idea is simply to regard this as a family of vector bundles  $\{\sE|_{s\times C}\}_{s \in S}$  on  $C$  parametrized by  $S$.
Namely we consider the opposite classification map
\[
S \dashrightarrow  \text{(moduli of  bundles on $C$)}, \quad s \mapsto [\sE|_{s\times C}]
\]
instead of the Grassmannian embedding $\Phi_{|E|}$, and describe the image of $S$ as a degeneracy locus of certain vector bundles.
For example, Mukai \cite{Muk96} described generic $S$ of genus $11$ as a non-abelian Brill-Noether locus.

Now we explain our VBAC method for genus 13 in detail from curve side.
Let $C$ be a curve of genus $g = 3$ and $\xi$ a line bundle of odd degree on $C$.
$\sN^- = \SU_C(2,\xi)$ denotes the moduli space of stable vector bundles of rank two with determinant $\xi$.
Then $\sN^-$ is a Fano variety of dimension $6$ with $\Pic(\sN^-) \simeq \bZ$.
Let $\sU$  be the (normalized) universal bundle on $\sN^- \times C$.
Here we normalize its first Chern class as
\[
c_1(\sU) = \alpha \otimes 1 + 1\otimes df  \in H^2(\sN^- \times C),
\]
where $d$ is the degree of $\xi$ and $f \in H^2(C)$ is the fundamental class of $C$.
For $p \in C$, $\sU_{p}$ denotes the vector bundle $\sU|_{\sN^- \times \{p\}}$, which is generated by global sections. The space of global sections $H^0(\sU_p)$ has dimension $2^g=8$ (see Section~\ref{section:N-}).

\begin{theorem}[K3 surfaces with $g=13$ in $\sN^-$]\label{theorem:CI_K3_in_N-}
Let $p_1$, $p_2 \in C$ be two general points.
Then the complete intersection variety in $\sN^-$ with respect to the vector bundles $\sU_{p_1}$ and $\sU_{p_2}$ is a primitively polarized K3 surface of genus $13$.

Namely, for general sections $s_i \in H^0(\sU_{p_i})$, the common  zero locus $(s_1)_0 \cap (s_2)_0$ is a nonsingular K3 surface and $h \coloneqq \alpha|_S$ is a primitive ample divisor  with $(h^2) =24$.
\end{theorem}

We conversely show that a general polarized K3 surface $(S,h)$ of genus $13$ is obtained by the above construction:
For the polarized K3 surface $(S,h)$ in Theorem~\ref{theorem:CI_K3_in_N-}, the restriction $\sU_p|_S$ ($p \in C$) is a vector bundle of rank two with Mukai vector $( \ch_0, \ch_1,\ch_2+\ch_0 ) = (2,h,6)$.
For a general polarized K3 surface $(S,h)$ of genus $13$, the moduli space $T=\sM_S(2,h,6)$ of stable vector bundles with Mukai vector $(2,h,6)$ is a K3 surface \cite{Muk87}.
This K3 surface $T$ admits a natural polarization $\hat h$ of degree $6$ \cite{Muk87} (cf.\ \cite[Theorem 6.1.14]{HL10}), and hence $T$ is a complete intersection of hypersurfaces of degree $2$ and $3$ in $\bP^4$.

\begin{theorem}[Correspondence between the moduli spaces]\label{theorem:general_K3_in_N-}
Let $(S,h)$ be the polarized K3 surface as in Theorem~\ref{theorem:CI_K3_in_N-}.
 Then the following hold:
\begin{enumerate}
 \item $(S,h)$ is a \emph{general} polarized K3 surface with genus $13$.
 \item For each point $p \in C$, the restricted bundle $\sU_p|_S$ is a \emph{stable} vector bundle of rank $2$ in $T= \sM_S(2,h,6)$.
 \item The induced moduli map $C \to T$ sends $C$ to a nodal hyperplane section $C'$ of $T \subset \bP^4$.
 \item The induced rational map
\[
\Phi \colon \coprod_{(C,p_1,p_2) \in \sM_{3,2}}\bP_{*}(H^0(\sU_{p_1}))\times \bP_{*}(H^0(\sU_{p_2}))  \dashrightarrow \coprod_{(S,h) \in \sF_{13}}  T^{\vee}
\]
is birational.
Here $\sM_{3,2}$ is the moduli space of unordered $2$-pointed curves of genus three, $\sF_{13}$ the moduli space of primitively polarized K3 surfaces of genus $13$,  and $T^{\vee}$ the projective dual variety of $T \subset \bP^4$.
\end{enumerate}
\end{theorem}

Theorem~\ref{theorem:general_K3_in_N-} will be proved by constructing the inverse map:
we recover the triple $(C, s_1, s_2)$ from the pair of a general polarized K3 surface $(S,h) \in \sF_{13}$ and a nodal hyperplane section $C'$ of $T = \sM_S(2,h,6) \subset \bP^4$.
The normalization $C$ of $C'$ is a curve of genus $3$.
Since $T$ is a moduli space of vector bundles, there exists a rank two vector bundle $\sF$ on $S \times C$ corresponding to the map $C \to T$.
If $x \in S$ is a point, then $\sF|_{\{x\}\times C}$ is a rank two vector bundle on $C$.
We will prove that this bundle is stable of odd degree.
Thus we have a map $S \to \sN^-$ such that $\sU|_{S \times C} = \sF$ up to twist by a line bundle.

\begin{claim}
There are two points $p_1$ and $p_2$, and two sections $s_1 \in H^0(\sU_{p_1})$ and  $s_2 \in H^0(\sU_{p_2})$ such that $S = (s_1)_0\cap (s_2)_0$.
\end{claim}

Assume for a moment that $C$ is a \emph{hyperelliptic} curve of genus $3$.
Then the double cover $C \to \bP^1$ defines a pencil of quadric forms $\sP=\langle q_1,q_2\rangle \subset S^2\bC^8$.
Conversely a pencil of quadric forms $\sP\subset S^2\bC^8$ is called \emph{simple} if the associated double cover
\[
C: y^2 = \disc (x_1q_1 - x_2q_2)
\]
is a \emph{smooth} hyperelliptic curve of genus $3$.

For a quadric form $q \in S^2\bC^8$, we denote by $\Gr(\bC^8,2,q)$ the orthogonal Grassmann variety of $2$-dimensional isotropic quotients of $\bC^8$.
By \cite{DR76}, $\sN^-$ is isomorphic to 
\[
\Gr(\bC^8,2,\sP)\coloneqq \cap_{q \in \sP} \Gr(\bC^8,2,q) = \Gr(\bC^8,2,q_1)\cap \Gr(\bC^8,2,q_2) \subset \Gr(\bC^8,2).
\]
The class $\alpha \in \Pic(\sN^-)$ is equivalent to the class of Pl\"ucker hyperplane section $\Gr(\bC^8,2,\sP) \subset \bP(\bigwedge^2 \bC^8)$.

If a quadric form $q \in S^2\bC^8$ is non-degenerate, then $\Gr(\bC^8,2,q)$ is a homogeneous space of the orthogonal group $\SO (q)$.
It admits two homogeneous vector bundles $\sS_{q,+}$, $\sS_{q,-}$ whose spaces of global sections are the representation spaces of the two spin representations of type $D_4$.
We call these bundles spinor bundles.

If a quadric form $q \in S^2\bC^8$ mildly degenerates, i.e., the rank of $q$ is $7$, then the variety $\Gr(\bC^8,2,q)$ admits a vector bundle $\sS_{q}$ whose space of global sections are the representation space of the spin representation of type $B_3$.
This corresponds to the folding of the Dynkin diagram $D_4$ to $B_3$:
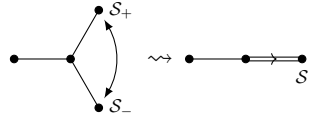
\[
\dynkin[edge length=.75cm,labels={,,\sS_+,\sS_-},involutions={43}]{D}{4} 
\rightsquigarrow
\dynkin[edge length=.75cm,labels={,,\sS},]{B}{3} 
\]
This defines a one-to-one correspondence between the restrictions of the spinor bundles to $\Gr(\bC^8,2,\sP)$ and the points on $C$, i.e., $C$ parametrizes restricted spinor bundles on $\Gr(\bC^8,2,\sP)$.
These restricted spinor bundles patch together to define a bundle on $\sN^- \times C$, which is isomorphic to the universal bundle $\sU$  up to twist by a line bundle from $C$.
Thus, for $p \in C$, $\sU_p$ is the restricted spinor bundle associated to $p \in C$ (see Proposition~\ref{proposition:N-_Gr}).

\begin{theorem}[K3 surfaces in $\Gr(\bC^8,2)$]\label{theorem:CI_K3_in_Gr(8,2)}
Let $\sP =\langle q_1, q_2\rangle \subset S^2\bC^8$ be a simple pencil of quadric forms, and $p_1$, $p_2 \in C$ general points.
Then the complete intersection variety in $\Gr(\bC^8,2,\sP)$ with respect to the vector bundles $\sU_{p_1}$ and $\sU_{p_2}$ is a primitively polarized $K3$ surface of genus $13$.

Namely, for general sections $s_i \in H^0(\sU_{p_i})$ ($i=1$, $2$),
 the common  zero locus $(s_1)_0 \cap (s_2)_0$ is a nonsingular K3 surface and $h \coloneqq \alpha|_S$ is a primitive ample divisor with $(h^2) =24$.
\end{theorem}

Similarly to the case of non-hyperelliptic curves, but replacing $T^\vee$ with $\sM_S(2, h, 6)$, we have the following:

\begin{theorem}[Correspondence between the moduli spaces]\label{theorem:general_K3_in_G(8,2)}
Let $S$ be a K3 surface as in Theorem~\ref{theorem:CI_K3_in_Gr(8,2)}.
\begin{enumerate}
 \item $(S,h)$ is a \emph{general} polarized K3 surface with genus $13$.
 \item The restriction of the universal quotient bundle of $\Gr(\bC^8,2)$ to $S$ is a \emph{stable} vector bundle of rank $2$ in $T= \sM_S(2,h,6)$.
 \item The induced rational map
\[
\Psi \colon \coprod_{(C,p_1,p_2) \in \sH_{3,2}}\bP_*(H^0(\sU_{p_1}))\times \bP_*(H^0(\sU_{p_2}))  \dashrightarrow  \coprod_{(S,h) \in \sF_{13}}  \sM_S(2,h,6)
\]
is birational.
Here $\sH_{3,2}$ is the moduli space of unordered $2$-pointed hyperelliptic curves of genus three.
\end{enumerate}
\end{theorem}

Thus, thanks to \cite{DR76}, we have the following new Grassmannian description via VBAC method:

\begin{corollary}\label{corollary:general_K3_in_G(8,2)}
 A general primitively polarized K3 surface of genus $13$ is an intersection of two ($7$-dimensional) contact homogeneous varieties  $\Ad$, $\Ad'$  of $\Spin(7)$ in the $12$-dimensional Grassmann variety  $\Gr(\bC^8, 2)$ (see diagram~\eqref{diag:S_Gr}).
\end{corollary}

\begin{remark}
Both  $\Ad$ and $\Ad'$  are  successive zero loci of homogeneous vector bundles in $\Gr(\bC^8, 2)$, though $S = \Ad \cap \Ad'$ itself is not.
Moreover, this description is a Grassmannian analogue of the Calabi-Yau $3$-folds  $\Gr(\bC^5, 2) \cap \Gr(\bC^5, 2)' \subset \bP^9$ of Gross-Popescu \cite{GP01} and Kanazawa \cite{Kan12}, and the double spinor Calabi-Yau $5$-folds  $\Sigma \cap \Sigma' \subset \bP^{15}$  of Manivel \cite{Man19}.
\end{remark}

\begin{remark}\label{rem:semi_rigid_bundles}
In \cite{Muk06}, a general polarized K3 surface $(S, h) \in \sF_{13}$ of genus $13$ is described in $\Gr(\bC^7,3)$ by using semi-rigid vector bundles in $\sM_S(3,h,4)$.
This description corresponds to the decomposition $g-1=12=3 \cdot 4$.
In our theorem, we use semi-rigid vector bundles in $ \sM_S(2,h,6)$, which corresponds to $g-1 =2\cdot 6$.
\end{remark}

The latter part of Theorem~ \ref{theorem:general_K3_in_G(8,2)} will be proved also by constructing the inverse map.
Namely, we recover the data $(q_1, q_2, s_1, s_2)$ from the pair of a general $(S, h) \in \sF_{13}$ and a general semi-rigid vector bundle $\sE \in \sM_S(2,h,6)$.
Recall that the moduli space $T\coloneqq \sM_S(2,h,6)$ is a complete intersection of a (unique) quadric hypersurface $\bQ^3$ and a cubic hypersurface in $\bP^4$.

\begin{observation}
Take a point $p' \in T \subset \bP^4$ and intersect $T$ with the tangent hyperplane $H$ of $\bQ^3$ at $p'$.
Then $C' \coloneqq H\cap T $ is a curve with a node at $p'$, whose normalization $C$ is a hyperelliptic curve of genus $3$.
\end{observation}
Applying this observation to the point $p'=[\sE]$, we obtain a hyperelliptic curve $C$ and its associated pencil $\sP\subset S^2 \bC^8$ of quadric forms.
Let $p'_1$, $p'_2\in C$ be the points over $p' = [\sE]$, and $p_i \in C$ the points that are the translations of $p'_i$ by the hyperelliptic involution $\iota$.
Then $H^0(\sU_{p_i})$ is the space of spinors corresponding to the points $p_i$, and the following holds:
\begin{claim}
There are two spinors  $s_1\in H^0(\sU_{p_1})$, $s_2\in H^0(\sU_{p_2})$ such that the K3 surface $S$ is the common zero locus $(s_1)_0\cap (s_2)_0$ in $\Gr(\bC^8,2,\sP)$.
\end{claim}

\begin{remark}
Let $p'_i\in C$ be the points over $p = [\sE]$ and $p_i$ the points corresponding to two sections $s_i$.
In the hyperelliptic case, we have $p_1+p_2+p'_1+p'_2 = K_C$.
For the non-hyperelliptic case, we do not know that this equality holds too.
\end{remark}

\subsection{Examples}

In general, the spinor bundles $\sU_{p_1}$, $\sU_{p_2}$ are defined on the different orthogonal Grassmann varieties $\Gr(\bC^8,2,q_1)$ and $\Gr(\bC^8,2,q_2)$, and hence our K3 surface $S$ is not a complete intersection variety with respect to \emph{homogeneous} vector bundles on $\Gr(\bC^8,2)$.
However, if $q_1$ and $q_2$ converge to one quadric form $q$, then $S$ is a complete intersection variety with respect to a homogeneous vector bundle $S^2 \sQ \oplus \sS_\pm \oplus \sS_\pm$ or $S^2\sQ \oplus \sS_\pm \oplus S_\mp$ in the homogeneous space $\Gr(\bC^8,2,q)$.
There are the following two cases:

\begin{example}\label{ex:complete_intersection}
\hfill
\begin{enumerate}
\item \label{ex:complete_intersection1}
The parity of spinors $s_1$, $s_2$ are the same;
in this case the K3 surface $S$ is a linear section of codimension $2$ of the Segre variety $\bP^1 \times \bP^1 \times \bP^1 \times \bP^1 \subset \bP^{15}$, and the polarization $h$ is the class of the hyperplane section.
\item \label{ex:complete_intersection2}
The parity of spinors $s_1$, $s_2$ are different;
in this case, the K3 surface $S$ is the minimal resolution of a sextic complete intersection variety $(2) \cap (3) \subset \bP^4$ with seven nodes $x_1$, \dots, $x_7$.
The polarization $h$ is $5 e_0-3 \sum_{i=1}^{7} e_i$,
where $e_0$ is the class of the hyperplane section, and $e_i$ are the exceptional lines over the points $x_i$.
\end{enumerate}

The latter case is nothing but the K3 surfaces discovered in \cite[Remark 1]{Muk06}, which are complete intersection varieties in a $G_2$-homogeneous variety with respect to a homogeneous vector bundle.
Note that the $4$-fold obtained as a blow-up of $\bP^4$ at seven points is the moduli space of parabolic vector bundles of rank two on a seven-pointed projective line \cite{Bau91}.
 
\end{example}

\begin{remark}
In \cite{Ben18}, there is  a (numerical) classification of $K$-trivial varieties of lower dimension that are complete intersection in classical homogeneous varieties with respect to homogeneous vector bundles.
The above two cases correspond to (ow3) and (ow3.2) in \cite{Ben18}.
\end{remark}

\subsection{Conjectural description for $g=19$}\label{sect:g_19}
 We expect that an analogous construction provides a description of general K3 surfaces of genus 19.

Let $C$ be a curve of genus $g \geq 2$, and $\xi$ a line bundle of degree one on $C$.
Then the moduli space $\sN=\SU_C(r,\xi)$ of stable vector bundles of rank $r$ with determinant $\xi$ is a Fano manifold of dimension $(r^2-1)(g-1)$.
Similarly to the rank two case, its Picard group is generated by an ample generator $\alpha$, and $-K_X=2 \alpha$.
Let $\sU$  be the normalized universal bundle on $\sN \times C$ with
\[
c_1(\sU) = \alpha \otimes 1 + 1\otimes f  \in H^2(\sN \times C).
\]
Then, for some values of $r$ and $g$, the bundle $\sU_p$ is globally generated.
This is the case, for instance, if $r=3$ and $g=2$ \cite{Ray82}.
The space of global sections $H^0(\sU_p)$ has dimension $r^g$.

Let $p_1$, $p_2 \in C$ be two points, and $s_i \in H^0(\sU_{p_i})$ general sections.
Then $(s_1)_0 \cap (s_2)_0$ is a smooth scheme of dimension $(r^2-1)(g-1)-2r$ whose canonical bundle is trivial.
In particuler, if $r=3$ and $g=2$, then $S \coloneqq (s_1)_0 \cap (s_2)_0$ is a $K$-trivial surface with a polarization $h \coloneqq \alpha|_S$.
We expect that this construction gives a K3 surface of $g=19$ and, moreover, that the following analogue of Theorem~\ref{theorem:general_K3_in_N-} hold:
\begin{enumerate}
 \item $(S,h)$ is a \emph{general primitively} polarized K3 surface with genus $19$.
 \item For each point $p \in C$, $\sU_p|_S$ is a \emph{stable} vector bundle of rank $3$ in $T= \sM_S(3,h,6)$.
 \item  The induced moduli map $C \to T$ sends $C$ to a nodal hyperplane section $C'$ of $T \subset \bP^3$ (Note that $T$ admits a natural polarization $\hat h$ of degree $4$, and hence is a quartic surface in $\bP^3$).
 \item The induced rational map
\[
\Phi \colon \coprod_{(C,p_1,p_2) \in \sM_{2,2}}\bP_{*}(H^0(\sU_{p_1}))\times \bP_{*}(H^0(\sU_{p_2}))  \dashrightarrow \coprod_{(S,h) \in \sF_{19}}  T^{\vee}
\]
is birational.
Here $\sM_{2,2}$ is the moduli space of unordered $2$-pointed curves of genus two, $\sF_{19}$ the moduli space of primitively polarized K3 surfaces of genus $19$,  and $T^{\vee}$ the projective dual variety of $T \subset \bP^3$.
\end{enumerate}

In a recent paper \cite{BBFM23b}, they gave a description of $\sN=\SU_C(3,\xi)$ by using the method of orbital degeneracy loci and confirmed that $S$ is indeed a K3 surface of genus $19$.

\subsection{Organization of the paper}
We briefly describe the contents of the present article.

In Section~\ref{section:equal_parity}, we study the case where the curve $C$ is hyperelliptic, $p_1=p_2$ and the parity of two spinors $s_1$ and $s_2$ are  the same (Example~\ref{ex:complete_intersection}~\ref{ex:complete_intersection1}).
We will show that $S$ is a linear section of $\bP^1 \times \bP^1 \times \bP^1 \times \bP^1 \subset \bP^{15}$.
As a consequence, we prove Theorems \ref{theorem:CI_K3_in_N-} and \ref{theorem:CI_K3_in_Gr(8,2)}.

In Section~\ref{section:N-}, we calculate several invariants, which are related to $\sN^-$ and the complete intersection variety $(s_1)_0 \cap (s_2)_0$ in $\sN^-$.
For example, we calculate the degrees of $(\pr_2)_*\sU$ and $(\pr_2)_*(\sU|_{S\times C})$.
One of the key ingredients of the proof is the structure of the cohomology ring of $\sN^-$, and almost all arguments in this section are irrelevant to the assumption that $g=3$.
Thus we will state and prove almost all results without assuming $g=3$.
A virtue of removing this assumption is that we can generalize our construction of K3 surfaces to that of Calabi-Yau varieties; we will construct families of polarized Calabi-Yau varieties of dimension $3g-7$.

In Section~\ref{section:cohomology_S}, we study the cohomology groups of vector bundles on $S$ and $\sN^-$, which will be used later.
The proof of the results in this section is firstly reduced to cohomological computations on  some homogeneous spaces.
Then we perform such computations by using the Borel-Weil-Bott theorem.

In Section~\ref{section:proof}, we prove the main theorems.
We first show Theorem~\ref{theorem:general_K3_in_G(8,2)} by using the results in previous sections.
Then Theorem~\ref{theorem:general_K3_in_N-} is proved by considering the specialization to the hyperelliptic case.

In Section~\ref{section:unequal_parity}, we study the case where the curve $C$ is hyperelliptic, $p_1=p_2$ but the parity of two spinors $s_1$ and $s_2$ are \emph{different} (Example~\ref{ex:complete_intersection}~\ref{ex:complete_intersection2}).

\subsection{Notations and convention}
Throughout this paper, we work over the complex number field $\bC$.
\begin{itemize}
\item For a smooth projective variety $X$ and $\Delta \in H^*(X,\bQ)$, we denote by
\[
\kappa_{\dim X}[ \Delta] \in H^{\dim X} (X,\bQ) \simeq \bQ
\]
the image of the natural projection $H^*(X,\bQ) \to H^{\dim X}(X,\bQ)$ or the corresponding rational number.
\item $\pr_i$ denotes the $i$-th projection from the product $\prod X_i$.
\item For a vector space $V$, $\bP(V)$ denotes the space of $1$-dimensional quotients of $V$.
$\bP_*(V)$ is the space of $1$-dimensional subspaces of $V$, i.e., $\bP_*(V)=\bP(V^\vee)$.
\item $\Gr(V,m)$ denotes the Grassmann variety of $m$-dimensional quotients of $V$.
$\sK$ and $\sQ$ are the universal subbundle or quotient bundle on this Grassmann variety, respectively.
\item $\Gr(V,m,q)$ is the orthogonal Grassmann variety with respect to a quadric form $q \in S^2V$.
\item $\sS_{q,\pm} =\sS_\pm$ are the spinor bundles.
\end{itemize}

For a K3 surface $S$,
\begin{itemize}
 \item $H^*(S,\bZ)$ is the integral cohomology ring of $S$.
 \item $\widetilde H(S,\bZ)$ is the Mukai lattice with the integral bilinear form $\langle \blank , \blank \rangle$.
 \item $T \colon H^*(S,\bZ) \to \widetilde H(S,\bZ)$ is defined as
 \[
 T(a,b,c) =(a,b,c)\sqrt{\td_S} = (a,b,c+a).
 \]
\item For a sheaf $\sE$ on $S$, its Mukai vector $v(\sE)$ is $T(\ch(\sE)) \in \widetilde H(S,\bZ)$.
\end{itemize}

\begin{acknowledgements}
The authors attended the Japanese-European Symposium on Symplectic Varieties and Moduli Spaces in May 2023 in Milan.  We warmly thank the organizers of this symposium for the invitation and their hospitality, and Laurent Manivel for his mini course there and also for discussions on the joint work \cite{BBFM23a,BBFM23b}.
\end{acknowledgements}

\section{Preliminary example and proofs of Theorems~\ref{theorem:CI_K3_in_N-} and \ref{theorem:CI_K3_in_Gr(8,2)}}
\label{section:equal_parity}

In this section, we describe the K3 surface $S$ in Example~\ref{ex:complete_intersection}~\ref{ex:complete_intersection1}, and prove Theorems~\ref{theorem:CI_K3_in_N-} and \ref{theorem:CI_K3_in_Gr(8,2)}.

Let $C$ be a hyperelliptic curve of genus $3$ and $\sP \coloneqq \langle q_1,q_2\rangle \subset S^2\bC^8$ the associated simple pencil of quadric forms.
Take a general point $p \in C$ and consider the restricted spinor bundle $\sU_p$ on $\Gr(\bC^8,2,\sP)$.
Thus $\sU_p$ is the restriction of a spinor bundle, say $\sS_+$, on $\Gr(\bC^8 ,2 ,q)$ for the associated quadric form $q \in S^2\bC^8$ (see Section~\ref{subsection:preliminaries_FK}).
Let $s_1$ and $s_2 \in H^0( \sU_p )$ be two general sections.
Since $H^0( \sU_p ) \simeq H^0(\sS_+)$, we can lift two sections $s_i$ to $\overline{s_i} \in H^0(\sS_+)$.

\begin{lemma}
Let the notation be as above.
Then the common zero locus $F\coloneqq (\overline{s_1})_0 \cap (\overline{s_2})_0 \subset \Gr(\bC^8,2,q)$ is the variety of lines on a $4$-dimensional smooth hyperquadric $\bQ^4$.
Equivalently, $F$ is the hyperplane section of the image of the Segre embedding $\bP^3 \times \bP^3 \subset \bP^{15}$.

Moreover, the restricted universal bundle $\sQ|_F$ is isomorphic to $\sO(1,0) \oplus \sO(0,1)$.
Here $\sO(a,b)$ is the line bundle $\pr_1 ^* \sO(a) \oplus \pr_2^*\sO(b)$ on $F$, where $\pr_i$ are the natural projections $F \to \bP^3$.
\end{lemma}

\begin{proof}
By the triality of $D_4$, the variety $\Gr(\bC^8,2,q)$ is isomorphic to $\Gr(H^0(\sS_+),2,q_+)$ for some $q_+ \in S^2H^0(\sS_+)$.
Thus $F = \Gr(W,2,q_+|_W)$, where $W =H^0(\sS_+)/\langle \overline{s_1},\overline{s_2}\rangle$.
This is the variety of lines on the $4$-dimensional hyperquadric $\bQ^4$.

Note that, via the identification $\Gr(\bC^8,2,q) \simeq \Gr(H^0(\sS_+),2,q_+)$, the universal bundle $\sQ$ on $\Gr(\bC^8,2,q)$ corresponds to a spinor bundle on $ \Gr(H^0(\sS_+),2,q_+)$.
Thus $Q|_F$ is the restriction of this spinor bundle to $\Gr(W,2,q_+|_W)$, which is isomorphic to the sum of two spinor bundles on $\Gr(W,2,q_+|_W)$ (cf.\ \cite[Theorem~1.4]{Ott88}).
This proves the lemma.
\end{proof}

\begin{theorem}
Let the notation be as above.
Then the common zero locus $S \coloneqq (s_1)_0 \cap (s_2)_0 \subset \Gr(\bC^8,2,\sP)$ is a codimension $2$ linear section of the Segre variety $\bP^1 \times\bP^1\times\bP^1\times\bP^1 \subset \bP^{15}$.
The polarization $h \coloneqq \alpha|_S$ is the class of the hyperplane section with respect to the Segre embedding $\bP^1 \times\bP^1\times\bP^1\times\bP^1 \subset \bP^{15}$.
\end{theorem}

\begin{proof}
 By the above lemma, a general quadric form $q' \in S^2 \bC^8$ restricts to a section $q'|_F \in H^0(S^2\sQ|_F) = H^0(\sO(2,0) \oplus\sO(1,1) \oplus\sO(0,2))$.
 Thus $S$ is the complete intersection variety in $\bP^3 \times \bP^3$ of type $(1,1) \cap (1,1) \cap (2,0) \cap (0,2)$.
 This variety is a codimension $2$ linear section of the Segre variety $\bP^1 \times\bP^1\times\bP^1\times\bP^1 \subset \bP^{15}$.
It is easy to see that $h$ is the class of the hyperplane section.
\end{proof}

In particular, the polarized surface $(S,h)$ in the above theorem is a polarized K3 surface of genus $13$.
Moreover the polarization $h$ is primitive by the Lefschetz hyperplane theorem.
By considering the specialization to this case, we have Theorems~\ref{theorem:CI_K3_in_N-} and \ref{theorem:CI_K3_in_Gr(8,2)}:

\begin{corollary}
Let $C$ be a hyperelliptic curve of genus $3$, and $\sP = \langle q_1,q_2\rangle \subset S^2\bC^8$ the associated simple pencil of quadric forms.
Take general two points $p_1$, $p_2 \in C$ and two general sections $s_i \in H^0(\sU_{p_i})$.
Then the common zero locus $S \coloneqq (s_1)_0 \cap (s_2)_0 \subset \Gr(\bC^8,2,\sP)$ is a nonsingular primitively polarized K3 surface of genus $13$.
\end{corollary}

\begin{corollary}
Let $C$ be a curve of genus $3$, and $p_1$, $p_2\in C$ be two general points.
Take two general sections $s_i \in H^0(\sU_{p_i})$.
Then the common zero locus $S \coloneqq (s_1)_0 \cap (s_2)_0 \subset \sN^-$ is a nonsingular primitively polarized K3 surface of genus $13$.
\end{corollary}

\section{The moduli spaces of rank two vector bundles on curves}\label{section:N-}

In this section, $C$ is a curve of genus $g \geq3$ and $\xi$ is a line bundle of odd degree on $C$.

\subsection{The moduli space $\sN^-$, its cohomology ring and intersection numbers}
Similar to the case of genus $3$, $\sN^- = \SU_C(2,\xi)$ is a Fano variety of dimension $3g-3$ with Picard group $\Pic(\sN^-) \simeq \bZ$.
The class $\alpha \in \Pic(X)$ denotes the ample generator of the Picard group of $X$.
Then $-K_{\sN^-} = 2 \alpha$. 
In what follows, we normalize the universal bundle on $\sN^-\times C$ as
\[
c_1(\sU) = \alpha \otimes 1 + 1 \otimes df,
\]
where $d$ is the degree of $\xi$ and $f \in H^2(C)$ is the fundamental class of $C$ as above.

It is known that there are two classes $\beta \in H^4(\sN^-)$ and $\gamma \in H^6(\sN^-)$, called Newstead classes, such that
\[
c_2(\sU) = \frac{d+1}{2}\alpha\otimes f + \sqrt{\gamma \otimes f} + \frac{\alpha^2-\beta}{4}\otimes 1.
\]
For $p \in C$, we have $c_2(\sU_p) = \frac{\alpha^2-\beta}{4}$.
Hence the Chern roots of $\sU_p$ are $\frac{\alpha\pm \sqrt{\beta}}{2}$.
In particular, we have
\begin{align*}
\ch(\sU_p)
 &= e^{(\alpha+\sqrt{\beta})/2}+e^{ (\alpha-\sqrt{\beta})/2}\\
 &= 2e^{\alpha/2} \cosh(\sqrt\beta/2)
\end{align*}
and
\begin{align*}
\td(\sU_p)
 &=  \frac{e^\alpha(\alpha^2-\beta)}{4(e^{\alpha/2}-e^{-\sqrt\beta/2})(e^{\alpha/2}-e^{\sqrt\beta/2})}.
\end{align*}
We can also write the Todd class of $\sN^-$ by using these classes $\alpha$ and $\beta$:
\[
\td(\sN^-) =e^\alpha \left(\frac{\sqrt{\beta}/2}{\sinh(\sqrt{\beta}/2)}\right)^{2g-2}.
\]

For $m$, $n$ and $p$ satisfying $m+2n+3p=3g-3$, intersection numbers among $\alpha$, $\beta$ and $\gamma$ are given as follows:
 \[
 (\alpha^m \beta^n \gamma ^p) =\frac{(-1)^n 2^{2g-2-p}g!m!  }{(g-p)!} b_{g-1-n-p}.
 \]
Here $b_i$ is the $2i$-th coefficient of the Laurent series expansion of $\dfrac{x}{\sin x}$:
 \[
 \frac{x}{\sin x} = \sum b_k x^{2k}.
 \]
For these facts, see, e.g., \cite{Muk03}.

For a power series $F$ and a rational number $k \in \bQ$, we have the following (\cite[p.\ 475]{Muk03}):
\[
\kappa_{3g-3}[e^{k\alpha}F(-\beta)]
=4^{g-1}k^{g} \Res _{x=0} \left[ \frac{F(x^2) \, dx}{x^{2g-2} \sin (kx) } \right].
\]
Similar arguments to the proof of the above formula yield the following:
\begin{proposition}[Residue expression for intersection numbers]
 For a power series $F$ in $\beta$ with coefficients in $\bQ$ and a rational number $k \in \bQ$, we have
 \[
\kappa_{3g-3}[e^{k\alpha}F(-\beta)\gamma]
=2^{2g-3}  k^{g-1} g \Res _{x=0} \left[ \frac{F(x^2) \, dx}{x^{2g-4} \sin (kx) } \right].
\]
Also,
\[
\kappa_{3g-3}[\alpha e^{k\alpha}F(-\beta)]
=4^{g-1}k^{g-1} \left(g\Res _{x=0} \left[ \frac{F(x^2) \, dx}{x^{2g-2} \sin (kx) } \right] - \Res _{x=0} \left[ \frac{k\cos(kx)F(x^2) \, dx}{x^{2g-3}\sin^2(kx) } \right]\right).
\]
\end{proposition}

\begin{proof}
Set $G(x) = \left(\dfrac{x}{\sin x}\right)' = \dfrac{1}{\sin x} - \dfrac{x\cos x}{\sin^2 x}$.
Then $G(x) = \sum 2k b_k x^{2k-1}$. 
Thus,
\begin{align*}
 \kappa_{3g-3}[\alpha e^{k\alpha}(-\beta)^n]
&=4^{g-1}k^{3g-2n-4} (3g-3-2n)b_{g-1-n}\\
&=4^{g-1}k^{3g-2n-4}( (g-1)b_{g-1-n} + (2g-2-2n)b_{g-1-n})\\
&=4^{g-1}k^{g-1} \left((g-1)\Res _{x=0} \left[ \frac{ x^{2n} \, dx}{x^{2g-2} \sin (kx) } \right] + \Res _{x=0} \left[ \frac{G(kx)x^{2n} \, dx}{x^{2g-2} } \right]\right).
\end{align*}
Therefore we have
\begin{align*}
 \kappa_{3g-3}[\alpha e^{k\alpha}F(-\beta)]
&=4^{g-1}k^{g-1} \left((g-1)\Res _{x=0} \left[ \frac{F(x^2) \, dx}{x^{2g-2} \sin (kx) } \right] + \Res _{x=0} \left[ \frac{G(kx)F(x^2) \, dx}{x^{2g-2} } \right]\right)\\
&=4^{g-1}k^{g-1} \left(g\Res _{x=0} \left[ \frac{F(x^2) \, dx}{x^{2g-2} \sin (kx) } \right] - \Res _{x=0} \left[ \frac{k\cos(kx)F(x^2) \, dx}{x^{2g-3}\sin^2(kx) } \right]\right).
\end{align*}

Similarly,
\begin{align*}
 \kappa_{3g-3}[e^{k\alpha}(-\beta)^n \gamma]
&=2^{2g-3} k^{g-1} g \Res _{x=0} \left[ \frac{ x^{2n} \, dx}{x^{2g-4} \sin (kx) } \right],
\end{align*}
and hence
 \[
\kappa_{3g-3}[e^{k\alpha}F(-\beta)\gamma]
=2^{2g-3} k^{g-1} g\Res _{x=0} \left[ \frac{F(x^2) \, dx}{x^{2g-4} \sin (kx) } \right].
\]
\end{proof}

\begin{corollary}
We have
\[
\kappa_{3g-3}\left[e^{k\alpha}F(-\beta)\left(-\frac{\alpha}{2} -\frac{\gamma}{ \beta}\right)\right]
=2^{2g-3}k^{g} \Res _{x=0} \left[ \frac{\cos(kx)F(x^2) \, dx}{x^{2g-3}\sin^2(kx) } \right]
\]
and, for $d \in \bZ$,
\[
\kappa_{3g-3}\left[e^{k\alpha}F(-\beta)\left(\frac{\gamma}{ \beta}+d\right)\right]
=2^{2g-3}k^{g-1}(2dk-g) \Res _{x=0} \left[ \frac{F(x^2) \, dx}{x^{2g-2}\sin(kx) } \right].
\]
\end{corollary}

\subsection{The case of hyperelliptic curves}
Assume for a moment that $C$ is a hyperelliptic curve of genus $g$, and let $\sP =\langle q_1,q_2 \rangle \subset S^2\bC^{2g+2}$ be the pencil of quadric forms associated to $C$.
Then $\sN^- \simeq \Gr(\bC^{2g+2},g-1,\sP) \coloneqq \Gr(\bC^{2g+2},g-1,q_1) \cap \Gr(\bC^{2g+2},g-1,q_2) \subset \Gr(\bC^{2g+2},g-1)$ \cite{DR76}.
In \cite[\S 2]{ST95}, the Chern classes of the restricted quotient bundle $\sQ|_{\sN^-}$ is described in terms of $\alpha$, $\beta$ and $\gamma$.
As a special case, we have:
\begin{proposition}[Chern classes of the universal bundle]
$c_1 (\sQ|_{\sN^-}) = \alpha$ and $c_2(\sQ|_{\sN^-}) =\frac{1}{2}\alpha^2 + \frac{1}{2}\beta$. 
\end{proposition}

\subsection{Hecke correspondence and associated bundles on $C$}\label{sect:hecke}
Let $p \in C$ be a point on $C$.
The \emph{Hecke correspondence} is, by definition, the following diagram:
\[
\xymatrix{
& \bP(\sU_p) \ar[ld]_-{\pi} \ar[rd]^-{\sigma}& \\
\sN^- & & \sN^+.
}
\]
Here $\sN^+$ is the moduli space of semistable vector bundles of rank two with fixed even determinant.
We briefly recall the construction of this diagram (see, e.g., \cite{Ses82}):
In the above diagram, $\pi \colon \bP(\sU_p) \to \sN^-$ is the natural projection from the projective bundle $\bP(\sU_p)$.
Take a point $x$ in $\bP(\sU_p)$.
This point corresponds to a surjection $\sU_p \otimes k(\pi(x)) \to \bC$.
This, in turn, defines the following exact sequence
\[
0 \to \sW \to \sU|_{\{\pi(x)\}\times C} \to \bC \to 0,
\]
where the last term is a skyscraper sheaf supported on $p$.
This kernel $\sW$ is a semistable vector bundle of rank two on $C$ with determinant $\xi(-p)$.
This defines the map $\sigma$.

By \cite[Corollary~5.16 and Lemma~7.4]{NR76},  a general $\sigma$-fiber defines conic in $\sN^-$, and dimensions of the $\sigma$-fibers are at most $g-1$.
In particular, $\sU_p$ is $g-1$-ample in the sense of \cite{Som78}.

Recall from \cite{Ray82,NR87,Bea88} that $\Pic(\sN^+) \simeq \bZ$ and its ample generator defines a map from $\sN^+$ to the dual projective space $\bP(H^0(2\Theta)^\vee)$ of the second order theta functions, and thus
\[
H^0(\sU_p) \simeq H^0(2\Theta)^\vee.
\]
Therefore the bundle $(\pr_{2})_*\sU$ on $C$ is a vector bundle of rank $2^g$.
In fact, we have the following:
\begin{proposition}[Degree of $(\pr_{2})_*\sU$]\label{proposition:VB_on_C}
The bundle $(\pr_{2})_*\sU$ on $C$ is a vector bundle of rank $2^g$ with degree $2^{g-1}(d-g+1)$.
\end{proposition}

\begin{proof}
Since $\sU_p$ is nef, $\bP (\sU_p)$ is a Fano manifold.
 The Kodaira vanishing theorem implies $h^i(\sU_p) =0$ for $i>0$.
 Hence the rank of  $(\pr_{2})_*\sU$ and its degree are calculated via the Grothendieck-Riemann-Roch theorem:
 \[
 \ch((\pr_2)_*\sU) = (\pr_2)_*( \ch(\sU) \cdot \td (T_{\pr_2})).
 \]

Thus we have
\begin{align*}
 \ch_0((\pr_2)_*\sU) &= \kappa_{3g-3}[\ch(\sU_p) \cdot \td (T_{\sN^-})] \\
 &= \kappa_{3g-3}\left[ 2e^{3\alpha/2} \cosh(\sqrt\beta/2) \left(\frac{\sqrt{\beta}/2}{\sinh(\sqrt{\beta}/2)}\right)^{2g-2}\right]\\
 &= 2\left(\frac{3}{2}\right)^{g}\Res_{x=0}\left[ \frac{\cos(x/2)\, dx}{\sin(3x/2)\sin^{2g-2}(x/2)}  \right]\\
 &=2^g.
\end{align*}
For the last residue, see Lemma~\ref{lemma:residue} below.

Similarly, we have
\[
\ch_1((\pr_2)_*\sU) = \kappa_{3g-3}[(\pr_1)_*\ch(\sU) \cdot \td (T_{\sN^-})].
\]
For $\ch(\sU)$, we use \cite[Lemma 2]{Zag95}:
Set $\delta \coloneqq -\frac{1}{2}\alpha\otimes f - \sqrt{\gamma \otimes  f}$.
Then, we have $\delta^2 = \gamma \otimes f $ , $\delta ^3 = 0$ and
\[
c(\sU) =\left(1+\frac{\alpha \otimes 1+1 \otimes d f+\sqrt\beta \otimes 1}{2}\right)\left(1+\frac{\alpha \otimes 1+ 1 \otimes df-\sqrt\beta \otimes 1}{2} \right)-\delta.
\]

Thus, 
\[
\ch(\sU) = e^{x_1}+e^{x_2} + \frac{e^{x_1}-e^{x_2}}{x_1-x_2}\delta - \left( \frac{e^{x_1}-e^{x_2}}{(x_1-x_2)^3} - \frac{e^{x_1}+e^{x_2}}{2(x_1-x_2)^2}\right) \delta^2,
\]
where $x_1 = (\alpha \otimes 1+1 \otimes d f+\sqrt\beta \otimes 1)/2$ and $x_2 = (\alpha \otimes 1+1 \otimes d f-\sqrt\beta \otimes 1)/2$.

Therefore, we have
\begin{align*}
&(\pr_1)_*\ch(\sU)\\
&= e^{\alpha/2}\left( -\alpha \frac{\sinh (\sqrt \beta /2)}{\sqrt \beta}
-2  \frac{\sinh (\sqrt \beta /2)}{\beta \sqrt \beta} \gamma +\frac{\cosh(\sqrt \beta /2)}{\beta}\gamma
+d \cosh(\sqrt\beta/2) \right)\\
&= e^{\alpha/2}\left(\frac{\sinh (\sqrt \beta /2)}{ \sqrt \beta/2}\left(-\frac{\alpha}{2}-\frac{\gamma}{\beta}\right)
+\cosh(\sqrt\beta/2)\left(\frac{\gamma}{\beta}+d \right) \right).
\end{align*}

Hence,
\begin{align*}
&\ch_1((\pr_2)_*\sU)\\
&= \kappa_{3g-3}[(\pr_1)_*\ch(\sU) \cdot \td (T_{\sN^-})]\\
& = \kappa_{3g-3}
\left[
e^{3\alpha/2}\left(\frac{\sqrt \beta/2}{\sinh (\sqrt \beta /2)}\right)^{2g-3}\left(-\frac{\alpha}{2}-\frac{\gamma}{\beta}\right)
+e^{3\alpha/2}\left(\frac{\sqrt \beta/2}{\sinh (\sqrt \beta /2)}\right)^{2g-2}\cosh(\sqrt\beta/2)\left(\frac{\gamma}{\beta}+d\right)
\right]\\
&=
\left(\frac{3}{2}\right)^{g-1}\Res_{x=0}\left[ \frac{3/2 \cos(3x/2) \, dx}{\sin^2 (3x/2) \sin^{2g-3}(x/2)} \right] +\left(\frac{3}{2}d-\frac{g}{2}\right)\left(\frac{3}{2}\right)^{g-1}\Res_{x=0}\left[ \frac{\cos(x/2)\, dx}{ \sin (3x/2) \sin^{2g-2}(x/2)} \right]\\
&=
2^{g-1}(d-g+1).
\end{align*}
For the last residue, see Lemma~\ref{lemma:residue}.
\end{proof}

Later we also need the following:
\begin{lemma}\label{lemma:cohomology_groups_on_N}
The following hold:
\begin{enumerate}
 \item $\chi(\sO_{\sN^-}(\alpha)) = 2^{g-1}(2^g-1)$.
 \item For $p \in C$, we have $\chi(\sU_p(-\alpha)) = 0$.
\end{enumerate}
\end{lemma}
\begin{proof}
By the Riemann-Roch theorem, we hvae
\begin{align*}
\chi(\sO_{\sN^-}(\alpha))  &= \kappa_{3g-3}[ ch(\sO_{\sN^-}(\alpha)) \cdot \td (\sN^-)] \\
&= \kappa_{3g-3}\left[ e^{2\alpha} \left(\frac{\sqrt{\beta}/2}{\sinh(\sqrt{\beta}/2)}\right)^{2g-2}\right]\\
&= 2^g\Res_{x=0}\left[ \frac{dx}{\sin(2x)\sin^{2g-2}(x/2)}  \right]\\
&=2^{g-1}(2^g-1).
\end{align*}

Also, we have
\begin{align*}
\chi(\sU_p(-\alpha))  &= \kappa_{3g-3}[\ch(\sU_p)\cdot ch(\sO_{\sN^-}(-\alpha)) \cdot \td (T_{\sN^-})] \\
&= \kappa_{3g-3}\left[ 2e^{\alpha/2} \cosh(\sqrt\beta/2) \left(\frac{\sqrt{\beta}/2}{\sinh(\sqrt{\beta}/2)}\right)^{2g-2}\right]\\
&= 2\left(\frac{1}{2}\right)^{g}\Res_{x=0}\left[ \frac{\cos(x/2)\, dx}{\sin^{2g-1}(x/2)}  \right]\\
&=0.
\end{align*}
For the residues, see Lemma~\ref{lemma:residue} below.
\end{proof}

\begin{remark}
The bundle $\sU_p$ embeds the variety $\sN^-$ into $\Gr(H^0(\sU_p),2)$.
In a recent paper \cite{BBFM23a}, they described the embedding $\sN^- \subset \Gr(H^0(\sU_p),2)$ for $g=3$ by the method of orbital degeneracy loci.
\end{remark}

\subsection{Vector bundles on $\sN^-$ and their complete intersections}

Fix a point $p \in C$.
Then we have a vector bundle $\sU_p = \sU|_{\sN^- \times \{p\}}$ of rank two on $\sN^-$, which is globally generated with $h^0(\sU_p) = 2^g$.
Take general sections $s_1 \in H^0(\sU_{p_1})$ and  $s_2 \in H^0(\sU_{p_2})$ ($p_1$, $p_2 \in C$), and set $Y \coloneqq (s_1)_0$ and $Z \coloneqq (s_1)_0 \cap (s_2)_0$.
If $g=3$, then $Z$ is a K3 surface of genus $13$.
The following generalizes this to higher genus cases:
\begin{theorem}[Fano and Calabi-Yau complete intersections in $\sN^-$]\label{theorem:CI_CY_in_N-}
The following hold:
\begin{enumerate}
 \item $Y$ is a smooth Fano variety of dimension $3g-5$;\label{theorem:CI_CY_in_N-_1}
 \item $Z$ is a smooth, simply connected, $K_Z$-trivial variety of dimension $3g-7$ with $\chi(\sO_Z) = 1+(-1)^{\dim Z}$;\label{theorem:CI_CY_in_N-_2}
 \item We have
\begin{align*}
& (\alpha|_Y )^{\dim Y} = 4^{g-2} \left((3g-3)!b_{g-1} +(3g-5)!b_{g-2}  \right),\\
& (\alpha|_Z )^{\dim Z} = 4^{g-3} \left((3g-3)!b_{g-1} + 2(3g-5)!b_{g-2} +(3g-7)!b_{g-3} \right).
\end{align*}
\label{theorem:CI_CY_in_N-_3}
\end{enumerate}
\end{theorem}

\begin{remark}[Degree of $Y$ and $Z$]
For $g=3$, $4$, $5$, $6$, \dots, we have
\begin{itemize}
 \item $\alpha|_Y ^{3g-5} = 72$, $13472$, $6913536$, $7477297152$, \dots
 \item $\alpha|_Z ^{3g-7} = 24$, $3840$, $1859968$, $1958247936$, \dots
\end{itemize}
\end{remark}

\begin{proof}[Proof of Theorem~\ref{theorem:CI_CY_in_N-}]
\ref{theorem:CI_CY_in_N-_1}
From adjunction and the Bertini theorem for globally generated vector bundles, it is enough to show the connectedness of $Y$.
By \cite[Proposition 1.16]{Som78}, $H^i(\sN^-,\bZ) \simeq H^i(Y, \bZ)$ for $i<2g-4$ and \ref{theorem:CI_CY_in_N-_1} follows.

\ref{theorem:CI_CY_in_N-_2}
From the Bertini theorem and adjunction, $Z$ is a disjoint union of smooth $K$-trivial varieties of equidimension $3g-7$.

By the Sommese vanishing theorem  \cite[Proposition 1.16]{Som78}, we have
\begin{equation}\label{eq:cohomology_on_Y}
H^i(Y,\bZ) \simeq H^i(Z, \bZ),
\end{equation}
if $i < 2g-6$.
If $g\geq 4$, then $H^0(Z,\bZ) = \bZ$, i.e.,  $Z$ is connected.
For $g=3$, the connectedness of $Z$ has been proved already.

We show the simple-connectedness of $Z$.
By the Okonek theorem \cite[Corollary 22]{Oko87}, we have $\pi_i(Y) \simeq \pi_i(Z)$ for $i < 2g-6$.
This proves $\pi_1(Z) = \{1\}$ for $g > 3$.
If $g=3$, then $Z$ is a K3 surface, and thus $Z$ is simply-connected.

Now we prove $\chi(\sO_Z) = 1+(-1)^{\dim Z}$.
Note that
\begin{align*}
\td(Z) &= (\td (\sN^-) \cdot \td(\sU_{p_1})^{-1} \cdot \td(\sU_{p_2})^{-1}) |_Z\\
&= \left. \frac{4^2}{e^\alpha (\alpha^2-\beta)^2} \left( e^{\alpha/2} -e^{\sqrt\beta/2} \right)^2 \left( e^{\alpha/2} -e^{-\sqrt\beta/2} \right)^2 \left(\frac{\sqrt\beta/2}{\sinh(\sqrt\beta/2)} \right)^{2g-2}\right|_Z.
\end{align*}
Thus
\begin{align*}
\chi(\sO_Z) &= \kappa_{3g-7} [\td Z]\\
&= \kappa_{3g-3}[\td(\sN^-)\td(\sU_{p_1})^{-1}\td(\sU_{p_2})^{-1} c_2(\sU_{p_1}) c_2(\sU_{p_2})]\\
&= \kappa_{3g-3} \left[  \left(e^\alpha +e^{-\alpha} -4 (e^{\alpha/2}+ e^{-\alpha/2})\cosh(\sqrt\beta/2) +4+2\cosh\sqrt \beta \right) \left(\frac{\sqrt\beta/2}{\sinh(\sqrt\beta/2)} \right)^{2g-2}\right]\\
&=(1+(-1)^{g+1})\Res_{x=0}\left[\frac{dx}{\sin^{2g-2}(x/2)\sin x }\right] -4\left(\left(\frac{1}{2}\right)^{g} -\left(-\frac{1}{2}\right)^{g} \right)\Res_{x=0}\left[\frac{\cos (x/2) \, dx}{\sin^{2g-1}(x/2) }\right]\\
&= 1+(-1)^{g+1}.
 \end{align*}
For the last residue, see Lemma~\ref{lemma:residue}.
These prove \ref{theorem:CI_CY_in_N-_2}.

\ref{theorem:CI_CY_in_N-_3}
Since $c_2(\sU_p) = \dfrac{\alpha^2-\beta}{4}$ for $p \in C$, we have
\[
(\alpha|_Y^{3g-5}) = \alpha^{3g-5} \cdot \left(\frac{\alpha^2-\beta}{4}\right) = 4^{g-2} \left((3g-3)!b_{g-1} + (3g-5)!b_{g-2}\right).
\]
Also,
\[
(\alpha|_Z^{3g-7}) = \left(\alpha^{3g-7} \cdot \left(\frac{\alpha^2-\beta}{4}\right)^2\right) = 4^{g-3} \left((3g-3)!b_{g-1} + 2(3g-5)!b_{g-2} +(3g-7)!b_{g-3} \right),
\]
which proves \ref{theorem:CI_CY_in_N-_3}.
\end{proof}

\begin{remark}
One can check $H^i(\sO_Z) =0$ for $0 < i < \dim Z$ if $g \geq 5$ as follows:
From \eqref{eq:cohomology_on_Y}, we have $H^i(\sO_Z) \simeq H^i(\sO_Y) =0$ if $i < 2g-6$.
If $g \geq 6$, then $\dim Z/2 < 2g-6$, and thus $H^i(\sO_Z) =0$ for $0 < i < \dim Z$.
If $g=5$, then $H^i(\sO_Z) \simeq H^i(\sO_Y) =0$ for $0 < i < 4$.
Since $\dim Z =8$ and $\chi(\sO_Z) =2$, we have $H^i(\sO_Z) =0$ for $0 < i < \dim Z$.

If $g=4$, then $Z$ is a simply-connected $K_Z$-trivial $5$-fold and the above argument only proves $H^1(\sO_Z) =  H^4(\sO_Z) =0$.
In fact, we can also check $H^2(\sO_Z) =  H^3(\sO_Z) =0$ (by using similar arguments in Section~\ref{section:cohomology_S}).

In summary, $Z$ is a Calabi-Yau manifold in the strong sense, i.e., $Z$ is a simply-connected $K_Z$-trivial variety with $H^i(\sO_Z) =0$ for $0 < i < \dim Z$.
\end{remark}

\subsection{Restrictions of $\sU$ and the associated bundles on $C$}
Let $Z$ be the Calabi-Yau variety as in Theorem~\ref{theorem:CI_CY_in_N-}.
Then $Z$ is defined by two sections $s_1$ and $s_2$.
Thus the natural evaluation map 
\[
\Ev \colon (\pr_2)_* \sU \to (\pr_2)_*(\sU|_{Z \times C})
\]
has non-trivial kernel at least over two points $p_1$ and $p_2$.
The following propositions ensure that the kernel of $\Ev$ has its support exactly on $p_1$ and $p_2$ under some conditions.
We only need Proposition~\ref{proposition:restricted_VB_on_C} for the proof of main theorems, but include Proposition~\ref{proposition:restricted_VB_on_C_general_genus} for the sake of completeness of our treatment.

\begin{proposition}[Degree of $(\pr_2)_*(\sU|_{Z \times C})$]
\label{proposition:restricted_VB_on_C_general_genus}
Assume $H^i(\sU|_{Z \times \{p\}}) = 0$ for $i > 0$ and for all $p\in C$.
Then $(\pr_2)_*(\sU|_{Z \times C})$ is a vector bundle of rank $2^g$ with degree $2^{g-1}(d-g+1) +2$.

If further $\Ev$ is generically an isomorphism,  then the kernel of $\Ev$ has its support exactly on $p_1$ and $p_2$.
\end{proposition}

\begin{proof}
From the Grothendieck-Riemann-Roch theorem,
\begin{align*}
& \ch_0((\pr_2)_*(\sU|_{Z \times C}))\\
&= \kappa_{3g-7}[ \ch(\sU|_{Z \times \{p\}}) \cdot \td(Z) ]\\
&= \kappa_{3g-3} \left[  2e^{\alpha/2} \cosh(\sqrt\beta/2) \cdot \frac{1}{e^\alpha} \left( e^{\alpha/2} -e^{\sqrt\beta/2} \right)^2 \left( e^{\alpha/2} -e^{-\sqrt\beta/2} \right)^2 \left(\frac{\sqrt\beta/2}{\sinh(\sqrt\beta/2)} \right)^{2g-2}\right]\\
&
\begin{multlined}
  = \kappa_{3g-3} \left[ 2 \left(e^{3\alpha/2} -4 e^{\alpha}\cosh(\sqrt\beta/2) +2e^{\alpha/2}(2+\cosh\sqrt \beta)  -4 \cosh(\sqrt\beta/2) +e^{-\alpha/2}  \right) \right.\\
\left. \times  \cosh(\sqrt\beta/2)\left(\frac{\sqrt\beta/2}{\sinh(\sqrt\beta/2)} \right)^{2g-2}\right]
\end{multlined}\\
&
\begin{multlined}
 =2\left(\frac{3}{2}\right)^g\Res_{x=0}\left[ \frac{\cos(x/2) \, dx}{\sin^{2g-2}(x/2)\sin(3x/2)} \right]
 - 4\cdot 2 \Res_{x=0}\left[ \frac{\cos^2(x/2) \, dx}{\sin^{2g-2}(x/2)\sin(x)} \right]\\
 + \left(\frac{1}{2}\right)^{g-2}\Res_{x=0}\left[ \frac{(2+\cos x)\cos(x/2) \, dx}{\sin^{2g-1}(x/2)} \right]
 + 2\left(-\frac{1}{2}\right)^g \Res_{x=0}\left[\frac{-\cos(x/2) \, dx}{\sin^{2g-1}(x/2)} \right]\\
\end{multlined}\\
 &= 2^g.
\end{align*}
Thus the rank of $(\pr_2)_*(\sU|_{Z \times C})$ is $2^g$.

Similarly, we have
\begin{align*}
 & \ch_1((\pr_2)_*(\sU|_{Z \times C}))\\
 &= \kappa_{3g-6}([\ch(\sU|_{Z \times C}) \cdot \pr^* \td(Z)] \\
 &= \kappa_{3g-3}[(\pr_1)_*\ch(\sU|_{Z \times C}) \td(\sN^-)\td(\sU_{p_1})^{-1}\td(\sU_{p_2})^{-1} c_2(\sU_{p_1}) c_2(\sU_{p_2})]\\
 &=\mu_{1,0}+\mu_{-1,0}+4\mu_{0,0}-4\mu_{1/2,1/2}+2\mu_{0,1}-4\mu_{-1/2,1/2},
\end{align*}
where
\begin{align*}
&
\begin{multlined}
 \mu_{k,j} \coloneqq \kappa_{3g-3}\left[
 e^{(2k+1)\alpha/2} \cosh(j\sqrt\beta)\left(\frac{\sqrt\beta/2}{\sinh(\sqrt\beta/2)} \right)^{2g-2}  \frac{\sinh (\sqrt \beta /2)}{ \sqrt \beta/2}\left(-\frac{\alpha}{2}-\frac{\gamma}{\beta}\right)
\right]\\
+ \kappa_{3g-3}\left[
e^{(2k+1)\alpha/2} \cosh(j\sqrt\beta)\left(\frac{\sqrt\beta/2}{\sinh(\sqrt\beta/2)} \right)^{2g-2}  \cosh(\sqrt\beta/2)\left(\frac{\gamma}{\beta}+d \right)
\right]
\end{multlined}\\
&
\begin{multlined}
 =\left(\frac{2k+1}{2}\right)^{g}\Res_{x=0}\left[ \frac{\cos (jx) \cos((2k+1)x/2)\, dx}{\sin^2((2k+1)x/2)\sin^{2g-3}(x/2)}\right]\\
+\frac{(2k+1)d-g}{2}\left(\frac{2k+1}{2}\right)^{g-1}\Res_{x=0}\left[ \frac{\cos (jx) \cos(x/2)\, dx}{\sin((2k+1)x/2)\sin^{2g-2}(x/2)}\right].
\end{multlined}
\end{align*}

By Lemma~\ref{lemma:residue}, we have
\[
\mu_{1,0}+\mu_{-1,0}+4\mu_{0,0}-4\mu_{1/2,1/2}+2\mu_{0,1}-4\mu_{-1/2,1/2}=2^{g-1}(d-g+1)+2.
\]
\end{proof}

If $g=3$, then $Z$ is a K3 surface of $g=13$.
In this case, the following slightly generalizes the above proposition:
\begin{proposition}[Degree of $(\pr_2)_*(\sU|_{S \times C})$]
\label{proposition:restricted_VB_on_C}
 Let $(S,h)$ be a polarized K3 surface of genus $13$.
 Assume that $S$ is contained in $\sN^-$ for $g=3$ such that $\alpha|_S =h$ and $\beta|_S = -8$.
 Further assume that  $H^i(\sU|_{S \times \{p\}}) = 0$ for $i > 0$ and for all  $p \in C$.
 Then $(\pr_2)_*(\sU|_{S \times C})$ is a vector bundle of rank $8$ with degree $4d-6$.
\end{proposition}

\begin{proof}
From the Grothendieck-Riemann-Roch theorem,
\begin{align*}
& \ch((\pr_2)_*(\sU|_{S \times C})) =(\pr_2)_* ( \ch(\sU|_{S \times C})\pr_1^*\td(S))\\
 &= (\pr_2)_*( \left( 2  [S \times C] + \ch_1(\sU|_{S \times C}) +\ch_2(\sU|_{S \times C}) +\ch_3(\sU|_{S \times C})\right) \left( [S \times C] + 2[\{x\} \times C]\right)).
 \end{align*}
Thus,
\begin{align*}
 & \ch_0((\pr_2)_*(\sU|_{S \times C})) =4 + (\pr_2)_*(\ch_2(\sU|_{S \times C}))\\
 &= 4 + (\pr_2)_* \left.\left(\frac{1}{2}(\alpha \otimes df)^2- \frac{d+1}{2}\alpha\otimes f + \sqrt{\gamma \otimes f} + \frac{\alpha^2-\beta}{4}\otimes 1\right)\right|_{S\times C}\\
&=8.
\end{align*}

On the other hand,
\begin{align*}
&\ch_1((\pr_2)_*(\sU|_{S \times C}))\\
&= (\pr_2)_*(\ch_1(\sU)\cdot 2[\{x\}\times C]+\ch_3(\sU|_{S \times C}))\\
 &= 4d-6.
\end{align*}
This completes the proof.
\end{proof}

\subsection{Residues}
Below we summarize the residues that we have used:
\begin{lemma}[Residues]\label{lemma:residue}
 The following hold:
\begin{enumerate}
 \item $\Res_{x=0}\left[\dfrac{dx }{ \sin^{2n+1}x \sin (2x)}\right] = \dfrac{1}{2}$ if $n\geq 0$.
 \item $\Res_{x=0}\left[\dfrac{dx }{ \sin^{2n}x \sin (4x)}\right] = \dfrac{2^{n+1}-1}{4}$ if $n\geq 0$.
 \item $\Res_{x=0}\left[\dfrac{\cos x \, dx}{ \sin^{2n+1}x}\right] = 0$ if $n\neq 0$.
 \item $\Res_{x=0}\left[\dfrac{\cos x  \, dx}{ \sin^{2n}x \sin (3x)}\right] = \dfrac{2^{2n}}{3^{n+1}}$ if $n\geq 0$.
 \item $\Res_{x=0}\left[\dfrac{\cos^2 x \, dx}{ \sin^{2n}x \sin (2x)}\right] = 0$ if $n\neq 0$.
 \item $\Res_{x=0}\left[\dfrac{\cos x \cos (2x) \, dx}{ \sin^{2n+1}x}\right] = 0$ if $n > 1$.
 \item $\Res_{x=0}\left[\dfrac{\cos 3x \, dx}{ \sin^{2n-1}x \sin^2 (3x)}\right] = -\dfrac{2^{2n}(2n-1)}{3^{n+2}}$ if $n\geq 0$.
 \item $\Res_{x=0}\left[\dfrac{\cos x \cos (2x) \, dx}{ \sin^{2n-1}x \sin^2 (2x)}\right] =-\dfrac{1}{4}$ if $n> 0$.
\end{enumerate}
\end{lemma}

\begin{proof}
 We can deduce these equations by letting $y = \sin x$.
 For example,
 \[
 \Res_{x=0}\left[ \frac{\cos x \, dx}{\sin^{2n+1}x} \right] = \Res_{y=0}\left[ \frac{dy}{y^{2n+1}} \right] = 0
 \]
 if $n \neq 0$.
\end{proof}

\section{Cohomology groups on $S$ and $\sN^{-}$}\label{section:cohomology_S}
Assume that $C$ is a  hyperelliptic curve of  genus $g =3$ and $\sP =\langle q_1,q_2\rangle \subset S^2\bC^8$ the associated simple pencil of quadric forms.
Then $\sN^- \simeq \Gr(\bC^8,2,\sP)$.
Let $p_1$, $p_2 \in C$ be two general points, and $s_i\in H^0(\sU_{p_i})$ general sections.
Then, by Theorem~\ref{theorem:CI_K3_in_Gr(8,2)}, $(S,h) \coloneqq ((s_1)_0 \cap (s_2)_0, \alpha|_S)$ is a primitively polarized K3 surface of genus $13$.

In this section, we will prove the following:
\begin{theorem}\label{theorem:cohomology_S}
The following hold:
\begin{enumerate}
\item\label{theorem:cohomology_S_2} The restriction map $H^0(\sO_{\sN^-}(1)) \to H^0(\sO_S(1))$ is a surjective map.
\item \label{theorem:cohomology_S_3} $H^0(\sQ|_S) = \bC^8$ and $H^i(\sQ|_S) = 0$ for $i>0$.
\item \label{theorem:cohomology_S_4} The natural restriction map $H^0(\sQ) \to H^0(\sQ|_S)$ is an isomorphism.
\item \label{theorem:cohomology_S_5} $H^0(S^2\sQ|_S) = \bC^{34}$ and $H^i(S^2\sQ|_S) = 0$ for $i>0$.
\item \label{theorem:cohomology_S_6} The natural restriction map $H^0(S^2\sQ) \to H^0(S^2\sQ|_S)$ is a surjective map with two dimensional kernel $\sP$.
\item \label{theorem:cohomology_S_7} $\sQ|_S$ is simple and semi-rigid, i.e., $H^0(S^2\sQ|_S (-h)) =0$ and $H^1(S^2\sQ|_S (-h)) =\bC^2$.
\item \label{theorem:cohomology_S_8} For $z \in C$, $H^0(\sU_z|_S) = \bC^8$ and $H^i(\sU_z|_S) = 0$ for $i>0$.
\item \label{theorem:cohomology_S_9} For $z \in C$, the natural restriction map $H^0(\sU_z) \to H^0(\sU_z|_S)$ is an isomorphism for $z \neq  p_i$.
\item \label{theorem:cohomology_S_10} The natural restriction map $H^0(\sU_{p_i}) \to H^0(\sU_{p_i}|_S)$ has one dimensional kernel $\langle s_i\rangle$.
\end{enumerate}
\end{theorem}

By Theorem~\ref{theorem:cohomology_S}~\ref{theorem:cohomology_S_8} and the Serre duality, we have the following corollary, which will be needed to ensure the stability of $\sU_z|_S$.
\begin{corollary}\label{corollary:stability}
$H^0(\sU_z|_S (-h)) =0$.
\end{corollary}

The proof is divided into three parts (= Propositions~\ref{proposition:cohomology_S}, \ref{proposition:restriction_S2Q} and \ref{proposition:simple_semirigid}).
We deduce these propositions by reducing it to cohomological computations on homogeneous varieties.
Such computations on homogeneous varieties will be done by using the Borel-Weil-Bott theorem and the Littlewood-Richardson rule (see Section~\ref{subsect:cohomology_HB}).

\subsection{Preliminaries}\label{subsection:preliminaries_FK}

Here we recall some results from \cite{FK18}.
In the following, $\Sigma^{\alpha}$ denotes the Schur functor associated to a Young diagram $\alpha$.
We adhere the same convention as \cite{FK18}.
For example $\Sigma^{(a)}\sE = S^a\sE$  and  $\Sigma^{(1, \dots ,1)}\sE = \bigwedge^a\sE$ for $a \in \bN$.

Let $\sP$ be the pencil of quadric forms associated to $C$, then $\sN^- = \Gr(\bC^8,2,\sP)$.
We denote by $f \colon C \to \bP^1$ the hyperelliptic covering, and identify the image $\bP^1$ with the pencil of quadric forms $\sP$.
Corresponding to the $8$ branching points, there are exactly $8$ quadric forms in $\sP$ of corank one.
Let $q$ be one of these eight degenerate quadric forms, and $I$ be the image of $q \colon (\bC^{8})^{\vee} \to \bC^8$.
Then $q$ induces a non-degenerate quadric form $q' \in S^2I$.
Set
\[
\widetilde \Gr(\bC^8,2,q) \coloneqq \{(Q_1,Q_2) \in \Gr(\bC^8,2,q)\times \Gr(I,2,q') \mid \Ker(I \to Q_2) \subset \Ker(\bC^8 \to Q_1)\}.
\]
Then $\widetilde\Gr(\bC^8,2,q)$ admits two natural projections:
\[
\xymatrix{
&\widetilde \Gr(\bC^8,2,q) \ar[ld]_-{p_1} \ar[rd]^-{p_2}& \\
\Gr(\bC^8,2,q) && \Gr(I,2,q')
}
\]
By \cite[Lemma~2.2]{FK18}, the map $p_1$ is a birational map and $\sN^- \subset  \Gr(\bC^8,2,q) $ avoids the indeterminacy locus of $p_1^{-1}$.
Thus the inclusion map $i_q \colon \sN^- \to  \Gr(\bC^8,2,q) $ induces an inclusion map $\widetilde i_q \colon \sN^- \to  \widetilde\Gr(\bC^8,2,q) $.
We denote by  $ j_q$ the inclusion map $j_q \colon \Gr(\bC^8,2,q) \to \Gr(\bC^8,2)$ and  $\widetilde j_q$ the composite $j_q\circ p_1 \colon \widetilde\Gr(\bC^8,2,q)\to \Gr(\bC^8,2,q) \to \Gr(\bC^8,2)$.
For a non-degenerate quadric form $q \in \sP$, we simply set $ \widetilde\Gr(\bC^8,2,q) \coloneqq \Gr(\bC^8,2,q)$, $\widetilde i_q \coloneqq i_q$ and $\widetilde j_q \coloneqq j_q$.
Then, for arbitrary two points $q_1 \neq q_2 \in \sP$, we have the following diagram:
\begin{equation}\label{diag:N-_Gr}
\xymatrix{
\sN^- \ar[r]^-{\widetilde i_{q_1}} \ar[d]_-{\widetilde i_{q_2}}&\widetilde \Gr(\bC^8,2,q_1)  \ar[d]^-{\widetilde j_{q_1}}\\
\widetilde \Gr(\bC^8,2,q_2) \ar[r]_-{\widetilde j_{q_2}}& \Gr(\bC^8,2)
}
\end{equation}

If $q$ degenerates, we denote by $\widetilde\sS$ the bundle $p_2^* \sS$.
For a non-degenerate quadric form $q \in \sP$, we simply set $\widetilde\sS_\pm = \sS_\pm$.

\begin{proposition}[{\cite[Lemma 2.4,  Proposition 2.7, Proposition 3.3, Lemma 3.5]{FK18}}]\label{proposition:N-_Gr}
The following hold:
\begin{enumerate}
\item Let $q \in \sP$ be a quadric form:
\begin{itemize}
\item If $q$ does not degenerate, then $f^{-1}(q) = \{x_+, x_-\}$ and $\sU_{x_{\pm}} \simeq (\widetilde i_{q})^* \widetilde\sS_\pm$.
\item If $q$ degenerates, then $f^{-1}(q) = \{x\}$ and $\sU_{x} \simeq (\widetilde i_{q})^* \widetilde\sS$.
\end{itemize}
\item Diagram \eqref{diag:N-_Gr} is a fiber product, and we have the base-change isomorphism
\[
(\widetilde i_{q_1})_* L(\widetilde i_{q_2})^* \simeq   L(\widetilde j_{q_1})^* (\widetilde j_{q_2})_*
\]
as functors $D^b(\widetilde\Gr(\bC^8,2,q_2)) \to D^b(\widetilde \Gr(\bC^8,2,q_1))$.
\item If $q$ does not degenerate, then  $(\widetilde j_{q})_*\widetilde\sS_{\pm}$ admits the following resolution on $\Gr(\bC^8,2)$:
\[
0 \to \sO(-2)^{\oplus *} \to \sK(-1)^{\oplus *} \to\sK^\vee (-1)^{\oplus *} \to\sO^{\oplus *} \to (\widetilde j_{q})_*\widetilde\sS_{\pm}\to 0.
\]
\item If $q$ degenerates, then $(\widetilde j_{q})_*  \widetilde\sS$ is contained in the triangulated subcategory
\[
\langle \Sigma^\alpha\sK \rangle \subset D^b(\Gr(\bC^8,2)),
\]
where $\alpha$ runs
\[
\{(0),(1,1,1,1,1),(1,1,1,1,1,1),(2,1,1,1,1,1),(2,2,2,2,2,2)\}.
\]
More precisely, we have
\[
(\widetilde j_{q})_*  \widetilde \sS \in \langle \sO, \bigwedge^5\sK \simeq \sK^\vee(-1), \bigwedge^6\sK \simeq \sO(-1), \sK(-1), \sO(-2) \rangle.
\]
\end{enumerate}
\end{proposition}

\subsection{Koszul complexes}\label{subsection:Koszul}
Recall that the K3 surface $S$ is defined by two sections $s_i \in \sU_{p_i}$ ($i = 1$, $2$).
Set $M_i \coloneqq (s_i)_0 \subset \sN^-$ and $q_i \coloneqq f(p_i)$.
The bundle $\sU_{p_i}$ is the restriction of one of the spinor bundles, say $\sS_+$,  on $\Gr(\bC^8,2,q_i)$, and we have 
$H^0(\sU_{p_i}) \simeq H^0(\sS_+)$.
By an abuse of notation, we denote also by $s_i$ the corresponding spinor in $H^0(\sS_+)$, and set $L_i \coloneqq (s_i)_0 \subset  \Gr(\bC^8,2,q_i)$.
Then we have the following diagram whose squares are given by fiber products with expected dimensions:
\begin{equation}\label{diag:S_Gr}
  \xymatrix{
 S\ar[r]^-{s_1=0} \ar[d]_-{s_2=0}&M_2\ar[r]^-{q_1=0} \ar[d]_-{s_2=0}&L_2 \ar[d]_-{s_2=0}\\
M_1\ar[r]^-{s_1=0} \ar[d]_-{q_2=0}&\sN^- \ar[r]^-{q_1=0} \ar[d]_-{q_2=0} & \Gr(\bC^8,2,q_2)  \ar[d]_-{q_2=0}\\
L_1 \ar[r]^-{s_1=0} & \Gr(\bC^8,2,q_1) \ar[r]^-{q_1=0} & \Gr(\bC^8,2).
}
\end{equation}
Thus we have the following exact Koszul complexes ($i \neq j$):
\begin{align}
 &\bigwedge^{\bullet} \sU_{p_i}^\vee |_{M_j} \to  \sO_S \to 0;\label{sequence:M_S}\\
 &\bigwedge^{\bullet} \sU_{p_i}^\vee  \to  \sO_{M_i} \to 0;  \label{sequence:N_M}\\
 &\bigwedge^{\bullet} \sS_+^\vee  \to \sO_{L_i} \to 0;  \label{sequence:OG_L}\\
 &\bigwedge^{\bullet} S^2\sQ^\vee |_{L_i} \to \sO_{M_i} \to 0;  \label{sequence:L_M}\\
 &\bigwedge^{\bullet} S^2\sQ^\vee |_{\Gr(\bC^8,2,q_i)} \to  \sO_{\sN^-} \to 0;  \label{sequence:OG_N}\\
 &\bigwedge^{\bullet} S^2\sQ^\vee \to  \sO_{\Gr(\bC^8,2,q_i)} \to 0;  \label{sequence:Gr_OG}\\
 &\bigwedge^{\bullet} (\sU_{p_1}\oplus\sU_{p_2})^\vee  \to \sO_{S} \to 0; \label{sequence:N_S}\\
 &\bigwedge^{\bullet} (S^2\sQ \oplus S^2\sQ)^\vee \to  \sO_{\sN^-} \to 0.  \label{sequence:Gr_N}
 \end{align}

Note that, by the triality of $D_4$, we have $L_i \simeq \Gr(\bC^7,2,q')$ for some non-degenerate quadric form $q' \in S^2\bC^7$.
Via this isomorphism, we have:
\begin{itemize}
\item $\sQ|_{L_i} \simeq \sS_-|_{L_i} \simeq \sS' $,  where $\sS'$ is the spinor bundle on $\Gr(\bC^7,2,q')$,
\item $ \sS_+|_{L_i} \simeq \sQ'$, where $\sQ'$ is the universal quotient bundle on $\Gr(\bC^7,2,q')$.
\end{itemize}
Here $\sS_+$ is the bundle containing the section $s_i$.

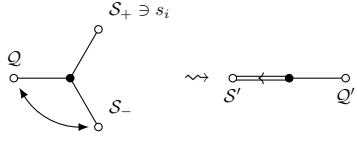
\[
\dynkin[%
edge length=.75cm, labels*={\sQ,,\sS_+\ni s_i,\sS_-}, involutions={14}]{D}{o*oo}
\rightsquigarrow
\dynkin[%
edge length=.75cm,backwards,labels={\sQ',,\sS'}]{B}{o*o}
\]

In the rest of this section, $\sO(1)$ is the polarization on each subvariety of $\Gr(\bC^8,2)$ corresponding to the Pl\"ucker embedding of $\Gr(\bC^8,2)$.
For example, $\sO_S(1) = \sO_S(h)$ and $\sO_{\sN^-}(1) = \sO_{\sN^-}(\alpha)$.

\subsection{Cohomology groups of $\sO_S$, $\sU_z|_S$ and $\sQ|_S$}
We start to prove Theorem~\ref{theorem:cohomology_S}.
Here we show the following:
\begin{proposition}\label{proposition:cohomology_S}
The following hold:
\begin{enumerate}
\item\label{proposition:cohomology_S_2} The restriction map $H^0(\sO_{\sN^-}(1)) \to H^0(\sO_S(1))$ is a surjective map.
\item \label{proposition:cohomology_S_3} $H^0(\sQ|_S) = \bC^8$ and $H^i(\sQ|_S) = 0$ for $i>0$.
\item \label{proposition:cohomology_S_4} The natural restriction map $H^0(\sQ) \to H^0(\sQ|_S)$ is an isomorphism.
\item \label{proposition:cohomology_S_5} For $z \in C$, $H^0(\sU_z|_S) = \bC^8$ and $H^i(\sU_z|_S) = 0$ for $i>0$.
\item \label{proposition:cohomology_S_6} For $z \in C$ with $z \neq p_i$, the natural restriction map $H^0(\sU_z) \to H^0(\sU_z|_S)$ is an isomorphism.
\item \label{proposition:cohomology_S_7} The natural restriction map $H^0(\sU_{p_i}) \to H^0(\sU_{p_i}|_S)$ has one dimensional kernel $\langle s_i\rangle$.
\end{enumerate}
\end{proposition}

To prove this proposition, we calculate some cohomology groups on $\sN^-$.
We start with an easy case:
\begin{lemma}\label{lemma_O_N}
For $-2 \leq m \leq 1$, we have $H^i(\sO_{\sN^-}(m))=0$ except for the following cases:
\begin{enumerate}
\item $H^0(\sO_{\sN^-}(1))=\bC^{28}$.
\item $H^0(\sO_{\sN^-})=\bC$.
\item $H^6(\sO_{\sN^-}(-2))=\bC$.
\end{enumerate}
\end{lemma}

\begin{proof}
These follow from the Kodaira vanishing theorem, the Serre duality theorem and Lemma~\ref{lemma:cohomology_groups_on_N}.
\end{proof}

Next we calculate $H^i(\sU_z(m))$ for $z \in C$:
\begin{lemma}\label{lemma:U}
Let $z \in C$ be a point.
Then, for $-2 \leq m \leq 0$, we have $H^i(\sU_z(m))=0$ except for $H^0(\sU_z)=\bC^8$.
\end{lemma}

\begin{proof}
Note that $\bP(\sU_z)$ is a Fano variety.
For $i>0$, we have
\[
H^i(\sU_z) = H^i(\bP(\sU_z), \sO_{\bP(\sU_z)}(1)) = 0
\]
by the Kodaira vanishing theorem.
By Proposition~\ref{proposition:VB_on_C}, we have $\chi(\sU_z)=8$ and hence the assertion for $m=0$ holds.

Let $\xi$ be the tautological divisor of $\bP(\sU_z)$.
Then
\[
H^i(\sU_z(-1)) = H^i(\bP(\sU_z), K_{\bP(\sU_z)}+3\xi).
\]
From Section~\ref{sect:hecke}, $\xi$ is spanned and the linear system $|\xi|$ defines a morphism $\sigma \colon \bP(\sU_z) \to \sN^+$ whose image has dimension $5$.
By the Kawamata-Viehweg vanishing theorem (see, e.g., \cite[Example 4.3.7]{Laz04a}), we have $H^i(\bP(\sU_z), K_{\bP(\sU_z)}+3\xi)=0$ for $i >1$.
Since the map $\sigma$ is of fiber type, we have $\Psef(\bP(\sU_z)) =\Nef(\bP(\sU_z)) = \langle \xi,\pi^*\sO(1)\rangle$.
Hence $H^0(\sU_z(-1))=0$.
By Lemma~\ref{lemma:cohomology_groups_on_N}, $\chi(\sU_z(-1))=0$ and hence the assertion for $m=1$ holds.

The assertion for $m=-2$ follows from the Serre duality theorem:
\[
h^i(\sU_z(-2)) = h^{6-i} (\sO(-2) \otimes \sU_z^\vee(2)) = h^{6-i} (\sU_z(-1)).
\]
\end{proof}

\begin{lemma}\label{lemma:U_Udual}
Let $z \in C$ be a point and $p \in C$ a \emph{general} point.
Then, for  $m=0$, $-1$, we have  $H^j(\sU_z \otimes \sU_{p}^\vee(m)) =0$ except for the following cases:
$H^j(\sU_{p} \otimes \sU_{p}^\vee) =\bC$ for $j=0$, $1$.
\end{lemma}

\begin{proof}
For $H^j(\sU_z \otimes \sU_{p}^\vee)$, the assertion is proved in \cite[Propositions 2.13, 2.14]{FK18}.
Thus we turn to prove the assertion for $H^j(\sU_z \otimes \sU_{p}^\vee(-1))$.
We first deal with the case $z \neq p$ and $z \neq \iota(p)$, where $\iota$ is the hyperelliptic involution of $C$.

Consider diagram~\eqref{diag:N-_Gr} for $q_1=f(z)$ and $q_2=f(p)$.
Then we have
\begin{align*}
H^j(\sU_z\otimes\sU_{p}^\vee (-1)) &= H^j(\sU_z\otimes\sU_{p} (-2))\\
&=H^j((\widetilde i_{q_1})_*(\sU_z\otimes\sU_{p}(-2)))\\
&=H^j(\widetilde\sS_1 \otimes (\widetilde i_{q_1})_*(\sU_{p}(-2))) &&(\text{$\widetilde\sS_1$ is the bundle whose restriction is $\sU_z$})\\
&=H^j(\widetilde\sS_1(-2) \otimes L(\widetilde j_{q_1})^* (\widetilde j_{q_2})_* \widetilde \sS_2) &&(\text{$\widetilde\sS_2$ is the bundle whose restriction is $\sU_{p}$})\\
&=H^j((\widetilde j_{q_1})_*\widetilde\sS_1 \otimes (\widetilde j_{q_2})_* \widetilde\sS_2 (-2)).
\end{align*}
By Proposition~\ref{proposition:N-_Gr}, $(\widetilde j_{q_i})_*\widetilde\sS_i$ is contained in 
\[
\langle \Sigma^\alpha\sK \rangle \subset D^b(\Gr(\bC^8,2))
\]
where
\[
\alpha \in \{(0),(1,1,1,1,1),(1,1,1,1,1,1),(2,1,1,1,1,1),(2,2,2,2,2,2)\}.
\]
Thus $(\widetilde j_{q_1})_*\widetilde\sS_1 \otimes (\widetilde j_{q_2})_*\widetilde\sS_2(-2)$ is contained in
\[
\langle \Sigma^\alpha\sK \otimes \Sigma^\beta\sK \otimes \sO(-2)\rangle \subset D^b(\Gr(\bC^8,2)),
\]
where
\[
\alpha, \beta \in \{(0),(1,1,1,1,1),(1,1,1,1,1,1),(2,1,1,1,1,1),(2,2,2,2,2,2)\}.
\]

By Lemma~\ref{lemma:Sigma_K_Sigma_K_-2} below, we have
\[
H^j(\Sigma^\alpha\sK\otimes\Sigma^\beta\sK\otimes\sO(-2))=0
\]
for all $\alpha$, $\beta$ as above.
Thus the assertion for the case $z \neq p$ and $z \neq \iota(p)$  follows.

Next, we calculate $H^j(\sU_{p} \otimes \sU_{p}^\vee(-1))$ and $H^j(\sU_{p} \otimes \sU_{\iota(p)}^\vee(-1))$, which are isomorphic to $ H^j(\sU_{p} \otimes \sU_{p}(-2))$ and $H^i(\sU_{p}\otimes\sU_{\iota(p)}(-2))$.
From sequence~\eqref{sequence:OG_N}, the vanishing follows from the fact 
\[
H^j(\Gr(\bC^8,2,q_2),\sS_+\otimes\sS_{\pm}\otimes\sO(-2)\otimes \bigwedge^kS^2\sQ^\vee|_{\Gr(\bC^8,2,q_2)}) =0.
\]
This will be proved in Lemma~\ref{lemma:S2Q_S_S_-2}.
\end{proof}

\begin{remark}
The proof shows $H^i(\sU_x\otimes\sU_y(-2)) = 0$ if $f(x) \neq f(y)$.
\end{remark}

Next we show:

\begin{lemma}\label{lemma_U_U_U_-2}
Let $z \in C$ be a point and $p_1$, $p_2 \in C$ \emph{general} points.
Then $H^i(\sU_z \otimes\sU_{p_1} \otimes \sU_{p_2} (-2)) = 0$ for all $i$.
\end{lemma}

By specializing to the case $p_1 =p_2$, we see that it is enough to show
\[
H^i(\sU_z \otimes\sU_{p} \otimes \sU_{p} (-2)) = 0
\]
for any $z \in C$ and  a general point $p \in C$.

Assume for a moment that $z \neq p$, $\iota(p)$.
To calculate the above cohomology group, we consider diagram~\eqref{diag:N-_Gr} for $q_1 = f(z)$ and $q_2=f(p)$.
Then, similar to the proof of Lemma~\ref{lemma:U_Udual}, we have
\[
H^i(\sU_z \otimes\sU_{p} \otimes \sU_{p} (-2)) =  H^i((\widetilde j_{q_1})_* \widetilde\sS_1 \otimes (\widetilde j_{q_2})_* (\widetilde\sS_2 \otimes \widetilde\sS_2) (-2)),
\]
where $\widetilde\sS_1$ and $\widetilde\sS_2$ are the spinor bundles whose restrictions to $\sN^-$ are $\sU_z$ and $\sU_p$ respectively.

\begin{lemma}\label{lemma:S_S}
Let $q \in S^2\bC^8$ be a non-degenerate quadric form, $ j \colon \Gr(\bC^8,2,q) \to \Gr(\bC^8,2)$ the natural inclusion, and $\sS$ one of the spinor bundles on  $\Gr(\bC^8,2,q)$.
Then $ j _*(\sS \otimes \sS)$ is contained in the triangulated subcategory
\[
\langle \Sigma^\gamma \sK \rangle \subset D^b(\Gr(\bC^8,2)),
\]
where $\gamma$ runs
\[
\begin{multlined}
\{
(0),(1),(2),(1,1,1),(2,1,1),(1,1,1,1,1),(2,1,1,1,1),(1,1,1,1,1,1),\\
 (2,1,1,1,1,1),(2,2,1,1,1,1),(2,2,2,2,1,1),(2,2,2,2,2,2)
 \}.
\end{multlined}
\]
\end{lemma}

\begin{proof}
We follow the proof of \cite[Proposition~3.3]{FK18}.
By \cite{Kap84}, $D^b(\Gr(\bC^8,2))$ has a full strong exceptional collection
\[
D^b(\Gr(\bC^8,2)) =\langle \Sigma^\alpha \sK \mid \text{$\alpha =(k_1,k_2,\dots k_6)$ with $2 \geq k_1 \geq \cdots \geq k_6\geq 0$}\} \rangle.
\]
The left dual exceptional collection is
\[
D^b(\Gr(\bC^8,2)) =\langle \Sigma^{\trans{\alpha}} \sQ \mid \text{$\alpha =(k_1,k_2,\dots k_6)$ with $2 \geq k_1 \geq \cdots \geq k_6\geq 0$}\} \rangle.
\]
For a coherent sheaf $\sF$, we have the spectral sequence:
\[
E_1^{p,q}=\bigoplus_{p=-|\alpha|}\Ext^q(\Sigma^{\trans \alpha} \sQ ,\sF)\otimes\Sigma^\alpha \sK \Rightarrow \sF.
\]
We apply this for $ j _*(\sS \otimes \sS)$:
\begin{align*}
 E_1^{p,q}&=\bigoplus_{p=-|\alpha|}\Ext^q(\Sigma^{\trans \alpha} \sQ , j _*(\sS \otimes \sS))\otimes\Sigma^\alpha \sK \\
 &=\bigoplus_{p=-|\alpha|}H^q(\Gr(\bC^8,2,q), \Sigma^{\trans \alpha} \sQ^\vee|_{\Gr(\bC^8,2,q)} \otimes \sS \otimes \sS) \otimes\Sigma^\alpha \sK\\
 &\Rightarrow  j _*(\sS \otimes \sS).
\end{align*}
By Lemma~\ref{lemma:Qdual_S_S} below, the cohomology groups
\[
H^q(\Gr(\bC^8,2,q), \Sigma^{\trans \alpha} \sQ^\vee|_{\Gr(\bC^8,2,q)} \otimes \sS \otimes \sS)
\]
are nonzero for the following values of $\trans\alpha$:
\[
\trans\alpha = (6,6), (6,4), (6,2), (6,1), (6,0), (5,1), (5,0), (3,1), (3,0), (1,1), (1,0), (0,0).
\]
Thus the assertion follows.
\end{proof}

\begin{proof}[Proof of Lemma~\ref{lemma_U_U_U_-2}]
Assume first that $z \neq p$, $\iota(p)$.
By Lemmas \ref{proposition:N-_Gr} and \ref{lemma:S_S}, $(\widetilde j_{q_1})_* \sS_1 \otimes (\widetilde j_{q_2})_* (\sS_2 \otimes \sS_2) (-2)$ is contained in $\langle \Sigma^\alpha \sK \otimes \Sigma^\beta \sK (-2) \rangle$, where 
\[
 \alpha \in  
 \{(0),(1,1,1,1,1),(1,1,1,1,1,1),(2,1,1,1,1,1),(2,2,2,2,2,2)\},
\]
and
\[
\begin{multlined}
\beta \in \{
(0),(1),(2),(1,1,1),(2,1,1),(1,1,1,1,1),(2,1,1,1,1),(1,1,1,1,1,1),\\
 (2,1,1,1,1,1),(2,2,1,1,1,1),(2,2,2,2,1,1),(2,2,2,2,2,2)
 \}.
\end{multlined}
\]
Lemma~\ref{lemma:Sigma_K_Sigma_K_-2} below shows $H^i( \Sigma^\alpha \sK \otimes \Sigma^\beta \sK(-2)) =0$ for all $i$.
This proves Lemma~\ref{lemma_U_U_U_-2} for $z \neq p$, $\iota(p)$.

Next assume $z = p$ or $\iota(p)$.
Consider the inclusion $\sN^- \subset \Gr(\bC^8,2,q)$ for $q= f(p)$.
By sequence~\eqref{sequence:OG_N}, it is enough to show
\[
H^i(\bigwedge^m (S^2\sQ^\vee |_{\Gr(\bC^8,2,q)}) \otimes\sS_+ \otimes\sS_+ \otimes \sS_{\pm} (-2)) = 0.
\]
This will be proved in Lemma~\ref{lemma:S2Q_S_S_S_-2} below.
\end{proof}

Now we can prove Proposition~\ref{proposition:cohomology_S}
\begin{proof}[Proof of Proposition~\ref{proposition:cohomology_S}]
\ref{proposition:cohomology_S_2} follows from sequence~\eqref{sequence:N_S} and Lemmas \ref{lemma_O_N}, \ref{lemma:U} and \ref{lemma:U_Udual}.
\ref{proposition:cohomology_S_5}, \ref{proposition:cohomology_S_6} and \ref{proposition:cohomology_S_7} follow from sequence~\eqref{sequence:N_S} and  Lemmas \ref{lemma:U}, \ref{lemma:U_Udual} and \ref{lemma_U_U_U_-2}.

Now we proceed to prove \ref{proposition:cohomology_S_3}, \ref{proposition:cohomology_S_4}.
As we have noted in Section~\ref{subsection:Koszul}, there are isomorphisms
\[
\sQ|_{L_i}  \simeq \sS' \simeq  \sS_-|_{L_i},
\]
where $\sS'$ is the spinor bundle on $L_i \simeq \Gr(\bC^7,2,q')$.
This isomorphism preserves the global sections:
\[
H^0(\sQ) \simeq  H^0(\sS') \simeq H^0(\sS_-).
\]
Note that $\sS_-|_S = \sU_z|_S$ for some $z \in C$.
Thus the map $H^0(\sQ) \to H^0(\sQ|_S)$ can be identified with $H^0(\sU_z) \to H^0(\sU_z|_S)$, which is an isomorphism.
Similarly, $H^i(\sQ|_S) \simeq H^i(\sU_z|_S)$ and the assertions hold.
\end{proof}

\subsection{Cohomology groups of $S^2\sQ$}
We study the restriction map $H^0(S^2\sQ) \to H^0(S^2\sQ|_S)$ and the cohomology groups $H^i(S^2\sQ|_S(m))$.
For this purpose, we calculate the cohomology groups on $\sN^-$, which are related to $S^2\sQ$ .

\begin{lemma}\label{lemmaS2Q_N}
For $-3 \leq m \leq 0$, we have $H^i(S^2\sQ|_{\sN^{-}}(m)) =0$ except for the following cases:
\begin{enumerate}
 \item $H^0(S^2\sQ|_{\sN^-}) =\bC^{34}$,
 \item $H^2(S^2\sQ|_{\sN^{-}}(-1)) =\bC$,
 \item $H^3(S^2\sQ|_{\sN^{-}}(-2)) =\bC^2$,
 \item $H^4(S^2\sQ|_{\sN^{-}}(-3)) =\bC$.
\end{enumerate}
Moreover the map $H^0(S^2\sQ) \to H^0(S^2\sQ|_{\sN^-})$ has two dimensional kernel $\sP = \langle q_1,q_2 \rangle$.

\end{lemma}

\begin{proof}
By sequence~\eqref{sequence:Gr_N}, the assertions follow from cohomology computations on $\Gr(\bC^8,2)$.
The required computations will be done in Lemma~\ref{lemma_S2Q_S2Adual_S2Qdual}.
\end{proof}

\begin{lemma}\label{lemma:S2Q_U_N}
For $-3 \leq m \leq -1$, we have  $H^j(S^2\sQ|_{\sN^-} \otimes \sU_{p_i}(m)) =0$.
\end{lemma}

\begin{proof}
By sequence~\eqref{sequence:OG_N}, it is enough to show the following vanishings:
\[
H^j(S^2\sQ|_{\Gr(\bC^8,2,q_i)} \otimes \sS \otimes \bigwedge^k S^2(\sQ^\vee)|_{\Gr(\bC^8,2,q_i)} \otimes\sO(m)) =0
\]
for $-3 \leq m \leq -1$.
By using the resolution of $\sS$ (=Proposition~\ref{proposition:N-_Gr}), we can deduce these from the following vanishings on $\Gr(\bC^8,2)$:
\begin{itemize}
\item $H^j(S^2\sQ \otimes \bigwedge^k S^2(\sQ^\vee) (m)) =0$,
\item $H^j(S^2\sQ \otimes \bigwedge^k S^2(\sQ^\vee) (m)\otimes\sK^\vee(-1)) =0$,
\item $H^j(S^2\sQ \otimes \bigwedge^k S^2(\sQ^\vee) (m)\otimes\sK(-1)) =0$,
\item $H^j(S^2\sQ \otimes \bigwedge^k S^2(\sQ^\vee) (m)\otimes \sO(-2)) =0$.
\end{itemize}
These will be proved in Lemma~\ref{lemma:S2Q_S2Qdual_SigmaK_-m}.
\end{proof}

\begin{lemma}\label{lemma:S2Q_Udual_Udual}
For $-1 \leq m \leq 0$, we have  $H^j(S^2\sQ|_{\sN^-} \otimes \sU_{p_1}^\vee \otimes\sU_{p_2}^\vee(m) )=0$.
\end{lemma}

\begin{proof}
Similarly to the proof of Lemma~\ref{lemma:U_Udual}, we have
\[
H^j(S^2\sQ|_{\sN^-} \otimes \sU_{p_1}^\vee \otimes\sU_{p_2}^\vee(m)) = H^j(S^2\sQ(m-2) \otimes (\widetilde j_{q_1})_*\widetilde\sS \otimes(\widetilde j_{q_2})_*\widetilde\sS).
\]
Recall that $(\widetilde j_{q_i})_*\widetilde\sS \in \langle \Sigma^\alpha\sK \rangle$ for
\[\alpha \in  
\{(0),(1,1,1,1,1),(2,1,1,1,1,1),(2,2,2,2,2,2)\}.
\]
Thus $S^2\sQ(m-2) \otimes (\widetilde j_{q_1})_*\widetilde\sS \otimes(\widetilde j_{q_2})_*\widetilde\sS$ is contained in
\[
\langle S^2\sQ(m-2)\otimes  \Sigma^\alpha\sK \otimes \Sigma^\beta \sK \rangle,
\]
where $\alpha$ and $\beta$ run 
\[\alpha, \beta \in  
\{(0),(1,1,1,1,1),(2,1,1,1,1,1),(2,2,2,2,2,2)\}.
\]
By Lemma~\ref{S2Q_-2-m_SigmaK_Sigma_K}, the cohomology groups of $S^2\sQ(m-2) \otimes \Sigma^\alpha\sK \otimes \Sigma^\beta \sK$ are all zero for $m=-1$ and $0$.
Hence the assertion follows.
\end{proof}

\subsubsection{Surjectivity of $S^2H^0(\sQ) \to H^0(S^2\sQ|_S)$}
Here we will prove the following:
\begin{proposition}\label{proposition:restriction_S2Q}
We have $H^i(S^2\sQ|_{S})=0$ for $i>0$.
Also the natural restriction morphism
 \begin{equation}\label{eq:restriction}
 H^0(S^2\sQ|_{\sN^-}) \to H^0(S^2\sQ|_S)
\end{equation}
is an isomorphism.
\end{proposition}

\begin{proof}
By tensoring $S^2\sQ|_{\sN^-}$ with sequence~\eqref{sequence:N_S}, we have the following exact sequence:
\[
 S^2\sQ|_{\sN^-} \otimes \bigwedge^{\bullet} (\sU_{p_1}\oplus\sU_{p_2})^\vee \to S^2\sQ|_N- \to S^2\sQ|_{S} \to 0.
\]
By Lemmas~\ref{lemmaS2Q_N}, \ref{lemma:S2Q_U_N} and \ref{lemma:S2Q_Udual_Udual}, the map \eqref{eq:restriction} is an isomorphism if and only if the induced map
\[
\varphi: H^3(S^2\sQ|_{\sN^-}\otimes \bigwedge^2 \sU_{p_1}^\vee \otimes \bigwedge^2 \sU_{p_2}^\vee) \to H^2(S^2\sQ|_{\sN^-} \otimes \bigwedge^2 \sU_{p_1}^\vee) \oplus H^2(S^2\sQ|_{\sN^-} \otimes \bigwedge^2 \sU_{p_2}^\vee)
\]
is an isomorphism.
Note that, if this condition is satisfied, then $H^i(S^2\sQ|_{S})=0$ for $i>0$.

By tensoring $S^2\sQ|_{\sN^-} \otimes \bigwedge^2\sU_{p_j}^\vee$ ($j \neq i$) with sequence~\eqref{sequence:N_M}, we have
\[
S^2\sQ|_{\sN^-} \otimes \bigwedge^2\sU_{p_j}^\vee \otimes \bigwedge^{\bullet} \sU_{p_i}^\vee \to (S^2\sQ \otimes \bigwedge^2\sU_{p_j}^\vee)|_{M_i} \to 0.
\]
This induces the following map
\[
\varphi_i \colon H^3(S^2\sQ|_{\sN^-}\otimes \bigwedge^2 \sU_{p_1}^\vee \otimes \bigwedge^2 \sU_{p_2}^\vee) \to H^2(S^2\sQ|_{\sN^-} \otimes \bigwedge^2 \sU_{p_j}^\vee).
\]
The map $\varphi$ is nothing but $\varphi_2 \oplus \varphi_1$.
Thus Proposition~\ref{proposition:restriction_S2Q} follows from the following lemma.
\end{proof}

\begin{lemma}
 $\varphi_1$ and $\varphi_2$ are surjective morphisms with different kernels in $H^3(S^2\sQ|_{\sN^-}\otimes \bigwedge^2 \sU_{p_1}^\vee \otimes \bigwedge^2 \sU_{p_2}^\vee) $.
\end{lemma}

\begin{proof}
We identify $\bigwedge^2\sU_{p_i}$ with $\sO(1)$.

Consider the double complex
\[
\xymatrix{
(\bigwedge^{\bullet} S^2\sQ^\vee \otimes S^2\sQ(-2)) |_{\Gr(\bC^8,2,q_1)} \ar[r]&  S^2\sQ(-2)|_{\sN^-} \\
\bigwedge^{\bullet} S^2\sQ^\vee \otimes \bigwedge^{\bullet} S^2\sQ^\vee\otimes S^2\sQ(-2) \ar[r] \ar[u]& (\bigwedge^{\bullet} S^2\sQ^\vee\otimes S^2\sQ(-2)) |_{\Gr(\bC^8,2,q_2)} \ar[u]
}
\]
which is obtained from the lower-right square of diagram~\eqref{diag:S_Gr}.

By Lemma~\ref{lemma_S2Q_S2Adual_S2Qdual} below, the nontrivial cohomology groups of $\bigwedge^{\bullet} S^2\sQ^\vee \otimes \bigwedge^{\bullet} S^2\sQ^\vee\otimes S^2\sQ(-2)$ are
\begin{itemize}
\item $ H^6(\bigwedge^{1} S^2\sQ^\vee \otimes \bigwedge^{2} S^2\sQ^\vee\otimes S^2\sQ(-2)) =\bC$,
\item $H^6 (\bigwedge^{2} S^2\sQ^\vee \otimes \bigwedge^{1} S^2\sQ^\vee\otimes S^2\sQ(-2))=\bC$.
\end{itemize}
$A_1$ and $A_2$ denote these two groups respectively.
Then, by chasing the diagram, we have
\begin{itemize}
\item $H^5((\bigwedge^{2} S^2\sQ^\vee \otimes S^2\sQ(-2)) |_{\Gr(\bC^8,2,q_1)}) \simeq A_2 \simeq H^4((S^2\sQ^\vee \otimes S^2\sQ(-2)) |_{\Gr(\bC^8,2,q_2)})$,
\item $H^5((\bigwedge^{2} S^2\sQ^\vee \otimes S^2\sQ(-2)) |_{\Gr(\bC^8,2,q_2)}) \simeq A_1 \simeq H^4((S^2\sQ^\vee \otimes S^2\sQ(-2)) |_{\Gr(\bC^8,2,q_1)})$.
\end{itemize}
Moreover we have a natural isomorphism $H^3(S^2\sQ(-2)|_{\sN^-}) \simeq A_1 \oplus A_2$.

Similarly, we have the following double complex:
 \[
\xymatrix{
(\bigwedge^{\bullet} S^2\sQ^\vee \otimes S^2\sQ(-1)) |_{L_1} \ar[r]&  S^2\sQ(-1)|_{M_1} \\
\bigwedge^{\bullet} S^2\sQ^\vee|_{\Gr(\bC^8,2,q_1)} \otimes \bigwedge^{\bullet} \sS^\vee \otimes S^2\sQ(-1)|_{\Gr(\bC^8,2,q_1)} \ar[r] \ar[u]& \bigwedge^{\bullet} \sU_{p_1}\otimes S^2\sQ(-1) |_{\sN^{-}}. \ar[u]
}
\]
Here $\sS$ is the spinor bundle whose restriction to $\sN^-$ is $\sU_{p_1}$.
By Lemmas~\ref{lemma:S2Q_-1_S2Qdual_Sdual} and \ref{S2Qdual_S2Q_-m_L} below, the nontrivial cohomology groups of $(\bigwedge^{\bullet} S^2\sQ^\vee \otimes S^2\sQ(-1)) |_{L_1}$ and  $\bigwedge^{\bullet} S^2\sQ^\vee|_{\Gr(\bC^8,2,q_1)} \otimes \bigwedge^{\bullet} \sS^\vee \otimes S^2\sQ(-1)|_{\Gr(\bC^8,2,q_1)}$ are
\begin{itemize}
\item $H^2((S^2\sQ^\vee \otimes S^2\sQ(-1))|_{L_1}) =\bC$,
\item $H^4(\bigwedge^{2} S^2\sQ^\vee|_{\Gr(\bC^8,2,q_1)}  \otimes S^2\sQ(-1)|_{\Gr(\bC^8,2,q_1)})=\bC$,
\item $H^5(\bigwedge^{2} S^2\sQ^\vee|_{\Gr(\bC^8,2,q_1)} \otimes \bigwedge^{2} \sS^\vee \otimes S^2\sQ(-1)|_{\Gr(\bC^8,2,q_1)}) \simeq A_2$,
\item $H^4( S^2\sQ^\vee|_{\Gr(\bC^8,2,q_1)} \otimes \bigwedge^{2} \sS^\vee \otimes S^2\sQ(-1)|_{\Gr(\bC^8,2,q_1)})\simeq A_1$.
\end{itemize}
Then, by chasing the diagram, we see that the induced sequence
\[
0 \to H^1(S^2\sQ(-1)|_{M_1}) \to H^3(S^2\sQ(-2)|_{\sN^-}) \xrightarrow{\varphi_1} H^2(S^2\sQ(-1)|_{\sN^-}) \to 0
\]
is equivalent to the natural sequence
\[
0 \to A_1 \to A_1 \oplus A_2 \to A_2 \to 0.
\]
By symmetry, the sequence
\[
0 \to H^1(S^2\sQ(-1)|_{M_2}) \to H^3(S^2\sQ(-2)|_{\sN^-}) \xrightarrow{\varphi_2} H^2(S^2\sQ(-1)|_{\sN^-}) \to 0
\]
is also equivalent to the natural sequence
\[
0 \to A_2 \to A_1 \oplus A_2 \to A_1 \to 0,
\]
and the assertion holds.
\end{proof}

\subsubsection{Simplicity and semi-rigidity of $\sQ|_S$}
The following proposition ensures that $\sQ|_S$ is simple and semi-rigid.

\begin{proposition}\label{proposition:simple_semirigid}
 $H^0(S^2\sQ|_S (-1)) = 0$ and $H^1(S^2\sQ|_S(-1)) =\bC^2$.
\end{proposition}

\begin{proof}
 Note that $\chi(S^2\sQ(-1)|_S) =2$ by the Riemann-Roch theorem and that $h^0(S^2\sQ|_S (-1)) =h^2(S^2\sQ|_S (-1))$ by the Serre duality.
Thus it is enough to show $H^0(S^2\sQ|_S (-1)) = 0$.
By tensoring $S^2\sQ (-1)$ to sequence~\eqref{sequence:N_S}, we have
\[
\bigwedge^{\bullet} (\sU_{p_1}\oplus\sU_{p_2})^\vee \otimes S^2\sQ|_{\sN^-} (-1) \to S^2\sQ|_S (-1) \to 0.
\]

By Lemmas~\ref{lemmaS2Q_N}, \ref{lemma:S2Q_U_N} and \ref{lemma:S2Q_Udual_Udual}, the nontrivial cohomology groups of $\bigwedge^{\bullet} (\sU_{p_1}\oplus\sU_{p_2})^\vee \otimes S^2\sQ|_{\sN^-} (-1) $ are
\begin{itemize}
 \item $H^4(\bigwedge^{4} (\sU_{p_1}\oplus\sU_{p_2})^\vee \otimes S^2\sQ|_{\sN^-} (-1) )=\bC$,
 \item $H^3(\bigwedge^{2} \sU_{p_1}^\vee \otimes S^2\sQ|_{\sN^-} (-1) )=\bC^2$,
 \item $H^3(\bigwedge^{2} \sU_{p_2}^\vee \otimes S^2\sQ|_{\sN^-} (-1) )=\bC^2$,
 \item $H^2( S^2\sQ|_{\sN^-} (-1) )=\bC$.
\end{itemize}
Thus the assertion holds if the induced maps
\[
\psi_i \colon H^4(\bigwedge^{4} (\sU_{p_1}\oplus\sU_{p_2})^\vee \otimes S^2\sQ|_{\sN^-} (-1) ) \to H^3(\bigwedge^{2} \sU_{p_j}^\vee \otimes S^2\sQ|_{\sN^-} (-1) )
\]
are injective ($j \neq i$).
By tensoring sequence~\eqref{sequence:N_M} with $S^2\sQ|_{\sN^-}(-1) \otimes \bigwedge^2\sU_{p_j}$, we see that the injectivity of $\psi_i$ is equivalent to
\[
\begin{cases}
 H^k(S^2\sQ|_{M_i}(-2)) =0 \qquad & (k \neq 3),\\
 H^3(S^2\sQ|_{M_i}(-2)) =\bC.
\end{cases}
\]
By sequence~\eqref{sequence:L_M}, these can be deduced from cohomology computations on $L_i$.
The required computations on $L_i$ will be proved in Lemma~\ref{S2Qdual_S2Q_-m_L} below.
\end{proof}

\subsection{Cohomology groups of homogeneous vector bundles}\label{subsect:cohomology_HB}
Let $G$ be a simply connected semi-simple algebraic group and $P$ a maximal parabolic subgroup of $G$.
A homogeneous vector bundle on $G/P$ is constructed from a representation of $P$.
Among them, completely reducible bundles correspond to representations of the reductive part $G_0$ of $P$.
We identify the weight lattice of $G$ and $G_0$.

In the following, $\rho$ denotes the half of the sum of all positive roots of $G$.
A weight $\lambda$ is called singular if $(\lambda,\alpha) = 0$ for some positive root $\alpha$; it is called regular otherwise.
For a regular weight $\lambda$, its index $i_0$ is the number of positive roots $\alpha$ with $(\lambda,\alpha)<0$.
If $\lambda$ is a regular weight, then there exists an element $w_{\lambda}$ of the Weyl group such that $w_{\lambda}(\lambda)$ is contained in the fundamental Weyl chamber.

\begin{theorem}[Borel-Weil-Bott theorem]
Let $\sE$ a homogeneous vector bundle corresponding to a representation of $G_0$, and $\gamma$ its highest weight.
 Then the following hold:
\begin{enumerate}
 \item If $\gamma + \rho$ is singular, then $H^i(\sE) =0$ for all $i$.
 \item If $\gamma + \rho$ is regular of index $i_0$, then $H^i(\sE) =0$ except for $i = i_0$.
 $H^{i_0}(\sE)$ is the irreducible $G$-module with highest weight $w_{\gamma+\rho}(\gamma+\rho)-\rho$.
\end{enumerate}
\end{theorem}

\subsubsection{On $\Gr(\bC^8,2)$}
The Grassmann variety $\Gr(\bC^8,2)$ is a homogeneous variety of type $A_7$.
Let $\{(a_1,a_2;a_3,a_4,a_5,a_6,a_7,a_8) \in \bZ^8 \mid \sum a_i = 0\}$ be the root lattice of type $A_7$ and $\bZ^8/\bZ(1,1,\dots,1)$ the weight lattice.

Our root basis is $\{e_1-e_2, e_2-e_3, \dots , e_7-e_8\}$.
Thus the fundamental weights are $(1, \dots, 1, 0 \dots, 0)$.
The half of the sum of all positive roots is
\[
\rho = (7,5;3,1,-1,-3,-5,-7)/2 \equiv (7,6;5,4,3,2,1,0).
\]
The bundles $\sQ$ and $\sK$ on $\Gr(\bC^8,2)$ correspond to the weights $(1,0;0,0,0,0,0,0)$ and $(0,0;1,0,0,0,0,0)$ respectively.

\begin{lemma}\label{lemma:Sigma_K_Sigma_K_-2}
 Let $\alpha$ and $\beta$ be one of the following Young diagrams:
\[
\begin{multlined}
\alpha \in \{
(0),(1),(2),(1,1,1),(2,1,1),(1,1,1,1,1),(2,1,1,1,1),(1,1,1,1,1,1),\\
 (2,1,1,1,1,1),(2,2,1,1,1,1),(2,2,2,2,1,1),(2,2,2,2,2,2)
  \}.
\end{multlined}
 \]
and
\[
\beta \in \{(0),(1,1,1,1,1),(1,1,1,1,1,1),(2,1,1,1,1,1),(2,2,2,2,2,2)\}.
\]
Then $H^i(\Sigma^\alpha \sK \otimes \Sigma^\beta\sK (-2)) = 0$.
\end{lemma}

\begin{proof}
 By the Littlewood-Richardson rule, the functor $\Sigma^\alpha (\blank) \otimes \Sigma^\beta (\blank)$ is a direct sum of some Schur functors $\Sigma^\gamma (\blank)$ for the following $\gamma$:

\begin{longtable}{lllll}
\toprule
$(0,0,0,0,0,0)$&
\begin{tabular}{l}
 $(1,1,1,1,1,0)$
\end{tabular}&
$(1,1,1,1,1,1)$&
\begin{tabular}{l}
 $(2,1,1,1,1,1)$
\end{tabular}&
$(2,2,2,2,2,2)$\\ \midrule
$(1,0,0,0,0,0)$&
\begin{tabular}{l}
 $(2,1,1,1,1,0)$\\$(1,1,1,1,1,1)$
\end{tabular}
&$(2,1,1,1,1,1)$&
\begin{tabular}{l}
 $(2,2,1,1,1,1)$\\$(3,1,1,1,1,1)$
\end{tabular}
&$(3,2,2,2,2,2)$\\ \midrule
$(2,0,0,0,0,0)$&
\begin{tabular}{l}
$(3,1,1,1,1,0)$\\$(2,1,1,1,1,1)$
\end{tabular}&
$(3,1,1,1,1,1)$&
\begin{tabular}{l}
$(3,2,1,1,1,1)$\\$(4,1,1,1,1,1)$
\end{tabular}&
$(4,2,2,2,2,2)$\\ \midrule
$(1,1,1,0,0,0)$&
\begin{tabular}{l}
$(2,2,2,1,1,0)$\\$(2,2,1,1,1,1)$
\end{tabular}&
$(2,2,2,1,1,1)$&
\begin{tabular}{l}
$(2,2,2,2,1,1)$\\$(3,2,2,1,1,1)$
\end{tabular}&
$(3,3,3,2,2,2)$\\ \midrule
$(2,1,1,0,0,0)$&
\begin{tabular}{l}
$(3,2,2,1,1,0)$\\$(3,2,1,1,1,1)$\\$(2,2,2,1,1,1)$
\end{tabular}&
$(3,2,2,1,1,1)$&
\begin{tabular}{l}
$(3,2,2,2,1,1)$\\$(3,3,2,1,1,1)$\\$(4,2,2,1,1,1)$
\end{tabular}&
$(4,3,3,2,2,2)$\\ \midrule
$(1,1,1,1,1,0)$&
\begin{tabular}{l}
$(2,2,2,2,2,0)$\\$(2,2,2,2,1,1)$
\end{tabular}&
$(2,2,2,2,2,1)$&
\begin{tabular}{l}
$(3,2,2,2,2,1)$\\$(2,2,2,2,2,2)$
\end{tabular}&
$(3,3,3,3,3,2)$\\ \midrule
$(2,1,1,1,1,0)$&
\begin{tabular}{l}
$(3,2,2,2,2,0)$\\$(3,2,2,2,1,1)$\\$(2,2,2,2,2,1)$
\end{tabular}&
$(3,2,2,2,2,1)$&
\begin{tabular}{l}
$(3,2,2,2,2,2)$\\$(3,3,2,2,2,1)$\\$(4,2,2,2,2,1)$
\end{tabular}&
$(4,3,3,3,3,2)$\\ \midrule
$(1,1,1,1,1,1)$&
\begin{tabular}{l}
$(2,2,2,2,2,1)$
\end{tabular}&
$(2,2,2,2,2,2)$&
\begin{tabular}{l}
$(3,2,2,2,2,2)$
\end{tabular}&
$(3,3,3,3,3,3)$\\ \midrule
$(2,1,1,1,1,1)$&
\begin{tabular}{l}
$(3,2,2,2,2,1)$\\$(2,2,2,2,2,2)$
\end{tabular}&
$(3,2,2,2,2,2)$&
\begin{tabular}{l}
$(4,2,2,2,2,2)$\\$(3,3,2,2,2,2)$
\end{tabular}&
$(4,3,3,3,3,3)$\\ \midrule
$(2,2,1,1,1,1)$&
\begin{tabular}{l}
$(3,3,2,2,2,1)$\\$(3,2,2,2,2,2)$
\end{tabular}&
$(3,3,2,2,2,2)$&
\begin{tabular}{l}
$(3,3,3,2,2,2)$\\$(4,3,2,2,2,2)$
\end{tabular}&
$(4,4,3,3,3,3)$\\ \midrule
$(2,2,2,2,1,1)$&
\begin{tabular}{l}
$(3,3,3,3,2,1)$\\$(3,3,3,2,2,2)$
\end{tabular}&
$(3,3,3,3,2,2)$&
\begin{tabular}{l}
$(4,3,3,3,2,2)$\\$(3,3,3,3,3,2)$
\end{tabular}&
$(4,4,4,4,3,3)$\\ \midrule
$(2,2,2,2,2,2)$&
\begin{tabular}{l}
$(3,3,3,3,3,2)$
\end{tabular}&
$(3,3,3,3,3,3)$&
\begin{tabular}{l}
$(4,3,3,3,3,3)$
\end{tabular}&
$(4,4,4,4,4,4)$\\
\bottomrule
\end{longtable}

Thus, the bundle $\Sigma^\alpha \sK \otimes \Sigma^\beta\sK$ is a direct sum of some homogeneous vector bundles whose weights have $\gamma$ above as their tails.
Then the bundle $\Sigma^\alpha \sK \otimes \Sigma^\beta\sK (-2)$ is a direct sum of some homogeneous vector bundles whose weights $w$ have $\gamma$ as their tails and have $(-2,-2)$ as their heads.

For such a weight $\gamma$, the sum $w + \rho $ is singular if one of the following conditions is satisfied:
\begin{itemize}
\item $\gamma_j = j-1$ for some $j$,
\item $\gamma_j = j-2 $ for some $j$,
\item $\gamma_j+6-j =\gamma_k +6-k$ for some $j$ and $k$.
\end{itemize}

One of these conditions holds for each $\gamma$ above, and the assertion holds.
\end{proof}

\begin{lemma}\label{S2Q_-2-m_SigmaK_Sigma_K}
Let $\alpha$ and $\beta$ be one of the following Young diagram:
\[
\alpha, \beta \in  \{(0),(1,1,1,1,1),(2,1,1,1,1,1),(2,2,2,2,2,2)\}.
\]
Then $H^i(S^2\sQ(m-2)\otimes  \Sigma^\alpha\sK \otimes \Sigma^\beta \sK)=0$ for $m=-1$, $0$.
\end{lemma}
\begin{proof}
By the Littlewood-Richardson rule, $\Sigma^\alpha\sK \otimes \Sigma^\beta \sK$ is contained in $\langle\Sigma^\gamma\sK\rangle$, where $\gamma$ runs
\[
\begin{multlined}
\{(0),(1,1,1,1,1),(2,1,1,1,1,1),(2,2,2,2,2,2),(2,2,2,2,2),(2,2,2,2,1,1),\\(3,2,2,2,2,1),(3,3,3,3,3,2),(4,2,2,2,2,2),(3,3,2,2,2,2),(4,3,3,3,3,3),(4,4,4,4,4,4)\}.
\end{multlined}
\]
The weight corresponding to $\Sigma^\gamma \sK$ has $\gamma$ as its tail.
By tensoring $S^2\sQ(m-2)$, we add the head $(m,m-2)$ to this weight.
Thus we have  the following weights $w$.
The assertion follows from the Borel-Weil-Bott theorem.
\begin{longtable}{lll}
\toprule
$w$ & $w+\rho$ & index \\ \midrule
$(m,m-2;0,0,0,0,0,0)$ & $(7+m,4+m;5,4,3,2,1,0)$ & $\blank$\\
$(m,m-2;1,1,1,1,1,0)$ & $(7+m,4+m;6,5,4,3,2,0)$ & $\blank$\\
$(m,m-2;2,1,1,1,1,1)$ & $(7+m,4+m;7,5,4,3,2,1)$ & $\blank$\\
$(m,m-2;2,2,2,2,2,2)$ & $(7+m,4+m;7,6,5,4,3,2)$ & $\blank$\\
$(m,m-2;2,2,2,2,2,0)$ & $(7+m,4+m;7,6,5,4,3,0)$ & $\blank$\\
$(m,m-2;2,2,2,2,1,1)$ & $(7+m,4+m;7,6,5,4,2,1)$ & $\blank$\\
$(m,m-2;3,2,2,2,2,1)$ & $(7+m,4+m;8,6,5,4,3,1)$ & $\blank$\\
$(m,m-2;3,3,3,3,3,2)$ & $(7+m,4+m;8,7,6,5,4,2)$ & $\blank$\\
$(m,m-2;3,3,2,2,2,2)$ & $(7+m,4+m;8,7,5,4,3,2)$ & $\blank$\\
$(m,m-2;4,2,2,2,2,2)$ & $(7+m,4+m;9,6,5,4,3,2)$ & $\blank$\\
$(m,m-2;4,3,3,3,3,3)$ & $(7+m,4+m;9,7,6,5,4,3)$ & $\blank$\\
$(m,m-2;4,4,4,4,4,4)$ & $(7+m,4+m;9,8,7,6,5,4)$ & $\blank$\\
\bottomrule
\end{longtable}
\end{proof}

\begin{lemma}\label{lemma_S2Q_S2Adual_S2Qdual}
For $-3 \leq m \leq 0$, we have $H^i(S^2\sQ(m) \otimes \bigwedge^k(S^2\sQ^\vee \oplus S^2 \sQ^\vee))=0$  except for the following cases:
\begin{enumerate}
\item $H^0(S^2\sQ) =\bC^{36}$.
\item $H^0(S^2\sQ\otimes(S^2\sQ^\vee\oplus S^2 \sQ^\vee )) =\bC^{2}$.
\item $H^6(S^2\sQ(-1) \otimes \bigwedge^4(S^2\sQ^\vee \oplus S^2 \sQ^\vee))=\bC$.
\item $H^6(S^2\sQ(-2) \otimes \bigwedge^3(S^2\sQ^\vee \oplus S^2 \sQ^\vee))=\bC^2$.
\item $H^6(S^2\sQ(-3) \otimes \bigwedge^2(S^2\sQ^\vee \oplus S^2 \sQ^\vee))=\bC$.
\end{enumerate}
\end{lemma}

\begin{proof}
Note that $\sQ^\vee \simeq \sQ(-1)$.
Thus $S^2\sQ^\vee \simeq S^2\sQ(-2)$ and $\bigwedge^2S^2\sQ^\vee \simeq S^2\sQ(-3)$.
By the Littlewood-Richardson rule, the bundle $S^2\sQ(m)\otimes \bigwedge^k(S^2\sQ^\vee \oplus S^2 \sQ^\vee)$ is the direct sum of homogeneous vector bundles whose weights have the trivial tails and the following heads (with specified multiplicities).
Hence the assertion follows from the Borel-Weil-Bott theorem.
\begin{longtable}{lll}
\toprule
$k$ & heads of $w$ & index of $w+\rho$ \\ \midrule
$0$ & $(m+2,m)$ & $0$ ($m=0$)\\
$1$ & $2(m,m)$, $2(m+1,m-1)$, $2(m+2,m-2)$ & $0$ ($m=0$)\\
$2$ & $3(m-1,m-1)$, $5(m,m-2)$, $4(m+1,m-3)$, $(m+2,m-4)$ & $6$ ($m=-3$)\\
$3$ & $2(m-2,m+2)$, $8(m-1,m-3)$, $4(m,m-4)$, $2(m+1,m-5)$ & $6$ ($m=-2$)\\
$4$ & $3(m-3,m-3)$, $5(m-2,m-4)$, $4(m-1,m-5)$, $(m,m-6)$ & $6$ ($m=-1$)\\
$5$ & $2(m-4,m-4)$, $2(m-3,m-5)$, $2(m-2,m-6)$ & $\blank$\\
$6$ & $(m-4,m-6)$ & $\blank$\\
\bottomrule
\end{longtable}
\end{proof}

\begin{lemma}\label{lemma:S2Q_S2Qdual_SigmaK_-m}
$H^i(S^2\sQ\otimes \bigwedge^kS^2\sQ^\vee \otimes \Sigma^\gamma\sK (m))=0$ for $-3 \leq m \leq -1$  and $\gamma$ in
\[
\{ (0), (1,1,1,1,1), (2,1,1,1,1,1), (2,2,2,2,2,2) \}.
\]
\end{lemma}

\begin{proof}
By the Littlewood-Richardson rule, $S^2\sQ\otimes \bigwedge^kS^2\sQ^\vee  (m)$ is a direct sum of some homogeneous vector bundles whose weights have the trivial tails and  the following heads:
\begin{longtable}{ll}
\toprule
$k$ & heads of $S^2\sQ\otimes\bigwedge^k S^2\sQ^\vee(m) $  \\ \midrule
$0$ & $(m+2,m)$\\
$1$ & $(m+2,m-2)$, $(m+1,m-1)$, $(m,m)$\\
$2$ & $(m-1,m-1)$, $(m,m-2)$, $(m+1,m-3)$\\
$3$ & $(m-1,m-3)$\\
\bottomrule
\end{longtable}
Thus $S^2\sQ\otimes \bigwedge^kS^2\sQ^\vee \otimes \Sigma^\gamma\sK (m)$ is a direct sum of some homogeneous vector bundles whose weights have heads as above and have $\gamma$ as their tails.
Then, by adding $\rho$, we see that $w+\rho$ is singular for all these weights.
The assertion follows from the Borel-Weil-Bott theorem.
\end{proof}

\subsubsection{On $\Gr(\bC^8,2,q)$}

Let $q \in S^2\bC^8$ be a non-degenerate quadric form.
Then $\Gr(\bC^8,2,q)$ is a homogeneous variety of type $D_4$.
Let $\bZ^4$ be the root lattice of type $D_4$ and $\bZ^4+\bZ(1/2,1/2,1/2,1/2)$ the weight lattice.
Our root basis is $\{e_1-e_2, e_2-e_3, e_3-e_4, e_3+e_4\}$, and hence the fundamental weights are $(1,0,0,0)$, $(1,1,0,0)$, $(1/2,1/2,1/2,-1/2)$, $(1/2,1/2,1/2,1/2)$.
The half of the sum of all positive roots is
\[
\rho = (3,2,1,0).
\]

The bundles $\sQ|_{\Gr(\bC^8,2,q)}$, $\sO(1)$, $\sS_+$, $\sS_-$ correspond to the fundamental weights $(1,0,0,0)$, $(1,1,0,0)$, $(1/2,1/2,1/2,-1/2)$, $(1/2,1/2,1/2,1/2)$ respectively.

\begin{lemma}\label{lemma:S2Q_S_S_-2}
 $H^i (\bigwedge^kS^2\sQ^\vee|_{\Gr(\bC^8,2,q)} \otimes \sS_+ \otimes\sS_{\pm} (-2))=0$.
\end{lemma}

\begin{proof}
For ease of notation, $\sQ|_{\Gr(\bC^8,2,q)}$ is denoted by $\sQ$ along the proof.
The bundle $\bigwedge^kS^2\sQ^\vee(-2)$ are homogeneous vector bundles with the following weights:
\begin{itemize}
\item $(-2,-2,0,0)$ for $k=0$,
\item $(-2,-4,0,0)$ for $k=1$,
\item $(-3,-5,0,0)$ for $k=2$,
\item $(-5,-5,0,0)$ for $k=3$.
\end{itemize}
Also we see that $\sS_+ \otimes \sS_{\pm}$ are direct sums of homogeneous bundles with the following weights:
\begin{itemize}
 \item $(1,1,1,-1)$ and $(1,1,0,0)$ for $\sS_+ \otimes \sS_{+})$,
 \item $(1,1,1,0)$ for $\sS_+ \otimes \sS_{-}$.
\end{itemize}
Thus $\bigwedge^kS^2\sQ^\vee \otimes \sS_+ \otimes\sS_{\pm} (-2)$ are direct sums of some homogeneous vector bundles with the following weights.
The assertion follows from the Borel-Weil-Bott theorem.
\begin{longtable}{lll}
\toprule
$w$ & $w+\rho$ & index \\ \midrule
$(-1,-1,1,-1)$ & $(2,1,2,-1)$ & $\blank$\\ 
$(-1,-1,0,0)$ & $(2,1,1,0)$ & $\blank$\\ 
$(-1,-1,1,0)$ & $(2,1,2,0)$ & $\blank$\\ 
$(-1,-3,1,-1)$ & $(2,-1,2,-1)$ & $\blank$\\ 
$(-1,-3,0,0)$ & $(2,-1,1,0)$ & $\blank$\\ 
$(-1,-3,1,0)$ & $(2,-1,2,0)$ & $\blank$\\ 
$(-2,-4,1,-1)$ & $(1,-2,2,-1)$ & $\blank$\\ 
$(-2,-4,0,0)$ & $(1,-2,1,0)$ & $\blank$\\ 
$(-2,-4,1,0)$ & $(1,-2,2,0)$ & $\blank$\\ 
$(-4,-4,1,-1)$ & $(-1,-2,2,-1)$ & $\blank$\\ 
$(-4,-4,0,0)$ & $(-1,-2,1,0)$ & $\blank$\\ 
$(-4,-4,1,0)$ & $(-1,-2,2,0)$ & $\blank$\\ 
\bottomrule
\end{longtable}
\end{proof}

\begin{lemma}\label{lemma:Qdual_S_S}
Let $\trans\alpha = (\alpha_1,\alpha_2)$ ($6\geq \alpha_1 \geq \alpha_2 \geq  0$) be a Young diagram.
Then $\Sigma^{\trans\alpha}\sQ^\vee|_{\Gr(\bC^8,2,q)} \otimes\sS_+ \otimes \sS_+$ has non-trivial cohomology groups if and only if
\[
\trans\alpha \in \{(6,6),(6,4),(6,2),(6,1),(6,0),(5,1),(5,0),(3,1),(3,0),(1,1),(1,0),(0,0)\}.
\]
\end{lemma}

\begin{proof}
For ease of notation, $\sQ|_{\Gr(\bC^8,2,q)}$ is denoted by $\sQ$ along the proof.
The bundle $\Sigma^{\trans\alpha}\sQ^\vee \otimes\sS_+ \otimes \sS_+$ is the direct sum of homogeneous vector bundles with weights $(1-\alpha_2,1-\alpha_1,1,-1)$ and $(1-\alpha_2,1-\alpha_1,0,0)$.
The sum of these weights and $\rho$ are $(4-\alpha_2,3-\alpha_1,2,-1)$ and $(4-\alpha_2,3-\alpha_1,1,0)$.

The first weight is regular if and only if the following are all satisfied:
\begin{itemize}
 \item $4-\alpha_2 \neq \pm 2$, $\pm1$,
 \item $3-\alpha_1 \neq \pm2$, $\pm1$,
 \item $4-\alpha_2 \neq \pm(3-\alpha_1)$.
\end{itemize}
This is equivalent to $(\alpha_1,\alpha_2)=(6,4)$, $(6,1)$, $(6,0)$, $(3,1)$, $(3,0)$ or $(0,0)$.

The second weight is regular if and only if the following are all satisfied:
\begin{itemize}
 \item $4-\alpha_2 \neq \pm 1$, $0$,
 \item $3-\alpha_1 \neq \pm1$, $0$,
 \item $4-\alpha_2 \neq \pm(3-\alpha_1)$.
\end{itemize}
This is equivalent to $(\alpha_1,\alpha_2)=(6,6)$, $(6,2)$, $(6,0)$, $(5,1)$, $(5,0)$, $(1,1)$, $(1,0)$ or $(0,0)$.
These prove the assertion.
\end{proof}

\begin{lemma}\label{lemma:S2Q_S_S_S_-2}
 $H^i(\bigwedge^kS^2\sQ^\vee|_{\Gr(\bC^8,2,q)} \otimes \sS_+ \otimes \sS_+ \otimes \sS_{\pm} (-2)) =0$.
\end{lemma}
\begin{proof}
For ease of notation, $\sQ|_{\Gr(\bC^8,2,q)}$ is denoted by $\sQ$ along the proof.
The bundles $\sS_+ \otimes \sS_+ \otimes \sS_{\pm}$ are direct sums of homogeneous vector bundles with the following weights:
\begin{itemize}
 \item $(3/2,3/2,3/2,-3/2)$ and $(3/2,3/2,1/2,-1/2)$ for $\sS_+ \otimes \sS_+ \otimes \sS_{+}$,
 \item $(3/2,3/2,3/2,-1/2)$ and $(3/2,3/2,1/2,1/2)$ for $\sS_+ \otimes \sS_+ \otimes \sS_{-}$.
\end{itemize}
Thus $\bigwedge^kS^2\sQ^\vee \otimes \sS_+ \otimes \sS_+ \otimes \sS_{\pm} (-2)$ are direct sums of homogeneous vector bundles with the following weights.
The assertion follows from the Borel-Weil-Bott theorem.
\begin{longtable}{llll}
\toprule
$k$ & $w$ for $\bigwedge^kS^2\sQ^\vee \otimes \sS_+ \otimes \sS_+ \otimes \sS_{+} (-2)$ & $w+\rho$ & index \\ \midrule
$0$ &$(-1,-1,3,-3)/2$ & $(5,3,5,-3)/2$ & $\blank$\\ 
 &$(-1,-1,1,-1)/2$ & $(5,3,3,-1)/2$ & $\blank$\\ 
$1$ &$(-1,-5,3,-3)/2$ & $(5,-1,5,-3)/2$ & $\blank$\\ 
 &$(-1,-5,1,-1)/2$ & $(5,-1,3,-1)/2$ & $\blank$\\ 
$2$ &$(-3,-7,3,-3)/2$ & $(3,-3,5,-3)/2$ & $\blank$\\ 
 &$(-3,-7,1,-1)/2$ & $(3,-3,3,-1)/2$ & $\blank$\\ 
$3$ &$(-7,-7,3,-3)/2$ & $(-1,-3,5,-3)/2$ & $\blank$\\ 
 &$(-7,-7,1,-1)/2$ & $(-1,-3,3,-1)/2$ & $\blank$\\ 
 \bottomrule
\end{longtable}

\begin{longtable}{llll}
\toprule
$k$ & $w$ for $\bigwedge^kS^2\sQ^\vee \otimes \sS_+ \otimes \sS_+ \otimes \sS_{-} (-2)$ & $w+\rho$ & index \\ \midrule
$0$ &$(-1,-1,3,-1)/2$ & $(5,3,5,-1)/2$ & $\blank$\\ 
 &$(-1,-1,1,1)/2$ & $(5,3,3,1)/2$ & $\blank$\\ 
$1$ &$(-1,-5,3,-1)/2$ & $(5,-1,5,-1)/2$ & $\blank$\\ 
 &$(-1,-5,1,1)/2$ & $(5,-1,3,1)/2$ & $\blank$\\ 
$2$ &$(-3,-7,3,-1)/2$ & $(3,-3,5,-1)/2$ & $\blank$\\ 
 &$(-3,-7,1,1)/2$ & $(3,-3,3,1)/2$ & $\blank$\\ 
$3$ &$(-7,-7,3,-1)/2$ & $(-1,-3,5,-1)/2$ & $\blank$\\ 
 &$(-7,-7,1,1)/2$ & $(-1,-3,3,1)/2$ & $\blank$\\ 
 \bottomrule
\end{longtable}
\end{proof}

\begin{lemma}\label{lemma:S2Q_-1_S2Qdual_Sdual}
 $H^i(S^2\sQ|_{\Gr(\bC^8,2,q)}(-1)\otimes\bigwedge^kS^2\sQ^\vee|_{\Gr(\bC^8,2,q)} \otimes \bigwedge^l \sS^\vee_+ ) = 0$ except for the following cases:
\begin{enumerate}
 \item $H^4(S^2\sQ|_{\Gr(\bC^8,2,q)}(-1)\otimes\bigwedge^2S^2\sQ^\vee|_{\Gr(\bC^8,2,q)} ) = \bC$.
 \item $H^5(S^2\sQ|_{\Gr(\bC^8,2,q)}(-1)\otimes\bigwedge^2S^2\sQ^\vee|_{\Gr(\bC^8,2,q)} \otimes \bigwedge^2 \sS^\vee_+ ) = \bC$.
 \item $H^4(S^2\sQ|_{\Gr(\bC^8,2,q)}(-1)\otimes S^2\sQ^\vee|_{\Gr(\bC^8,2,q)} \otimes \bigwedge^2 \sS^\vee_+ ) = \bC$.
\end{enumerate}
\end{lemma}

\begin{proof}
The bundle $S^2\sQ|_{\Gr(\bC^8,2,q)}(-1)\otimes\bigwedge^kS^2\sQ^\vee|_{\Gr(\bC^8,2,q)} \otimes \bigwedge^l \sS^\vee_+$ is the direct sum of homogeneous vector bundles with the following weights (with multiplicity one).
The assertion follows from the Borel-Weil-Bott theorem.
\begin{longtable}{llll}
\bottomrule
$(l,k)$ & $w$ & $w+\rho$ & index \\ \midrule
$(0,0)$&$(1,-1,0,0)$ & $(4,1,1,0)$ & $\blank$\\
$(0,1)$&$(1,-3,0,0)$ & $(4,-1,1,0)$ & $\blank$\\
&$(0,-2,0,0)$ & $(3,0,1,0)$ & $\blank$\\
&$(-1,-1,0,0)$ & $(2,1,1,0)$ & $\blank$\\
$(0,2)$&$(0,-4,0,0)$ & $(3,-2,1,0)$ & $4$\\
&$(-1,-3,0,0)$ & $(2,-1,1,0)$ & $\blank$\\
&$(-2,-2,0,0)$ & $(1,0,1,0)$ & $\blank$\\
$(0,3)$&$(-2,-4,0,0)$ & $(1,-2,1,0)$ & $\blank$\\ \midrule
$(1,0)$&$(3,-1,1,-1)/2$ & $(9,3,3,-1)/2$ & $\blank$\\
$(1,1)$&$(3,-5,1,-1)/2$ & $(9,-1,3,-1)/2$ & $\blank$\\
&$(1,-3,1,-1)/2$ & $(7,1,3,-1)/2$ & $\blank$\\
&$(-1,-1,1,-1)/2$ & $(5,3,3,-1)/2$ & $\blank$\\
$(1,2)$&$(1,-7,1,-1)/2$ & $(7,-3,3,-1)/2$ & $\blank$\\
&$(-1,-5,1,-1)/2$ & $(5,-1,3,-1)/2$ & $\blank$\\
&$(-3,-3,1,-1)/2$ & $(3,1,3,-1)/2$ & $\blank$\\
$(1,3)$&$(-3,-7,1,-1)/2$ & $(3,-7,3,-1)/2$ & $\blank$\\ \midrule
$(2,0)$&$(0,-2,0,0)$ & $(3,0,1,0)$ & $\blank$\\
$(2,1)$&$(0,-4,0,0)$ & $(3,-2,1,0)$ & $4$\\
&$(-1,-3,0,0)$ & $(2,-1,1,0)$ & $\blank$\\
&$(-2,-2,0,0)$ & $(1,0,1,0)$ & $\blank$\\
$(2,2)$&$(-1,-5,0,0)$ & $(2,-3,1,0)$ & $5$\\
&$(-2,-4,0,0)$ & $(1,-2,1,0)$ & $\blank$\\
&$(-3,-3,0,0)$ & $(0,-1,1,0)$ & $\blank$\\
$(2,3)$&$(-3,-5,0,0)$ & $(0,-3,1,0)$ & $\blank$\\
\bottomrule
\end{longtable}
\end{proof}

\subsubsection{On $\Gr(\bC^7,2,q')$}
As we have already noted, $L_i \simeq \Gr(\bC^7,2,q')$ by the triality of $D_4$.
Here $q' \in S^2\bC^7$ is a non-degenerate quadric form.
This variety $\Gr(\bC^7,2,q')$ is a homogeneous variety of type $B_3$.
Let $\bZ^3$ be the root lattice of $B_3$ and $\bZ^3 + \bZ(1/2,1/2,1/2)$ the weight lattice.
Our root basis is $\{e_1-e_2, e_2-e_3, e_3\}$.
Thus the fundamental weights are $(1,0,0)$, $(1,1,0)$, $(1/2,1/2,1/2)$.
The half of the sum of all positive roots is
\[
\rho = (5,3,1)/2.
\]
Recall that  there are isomorphisms $\sQ|_{L_i} \simeq \sS' \simeq \sS_-|_{L_i} $ and $ \sS_+|_{L_i} \simeq \sQ'$.
Here $\sS'$ is the spinor bundle on $\Gr(\bC^7,2,q')$ and  $\sQ'$ is the universal quotient bundle on $\Gr(\bC^7,2,q')$.
Thus the bundles $\sO(1)$, $\sQ|_{L_i}$, $\sS_+|_{L_i}$, $\sS_-|_{L_i}$ correspond to the weights $(1,1,0)$, $(1/2,1/2,1/2)$, $(1,0,0)$, $(1/2,1/2,1/2)$ respectively.

\begin{lemma}\label{S2Qdual_S2Q_-m_L}
For $-2 \leq m \leq -1$, we have $H^j(\bigwedge^kS^2\sQ^\vee |_{L_i} \otimes S^2\sQ |_{L_i} (m) ) =0$ except for the following cases:
\begin{enumerate}
 \item $H^2(S^2\sQ^\vee |_{L_i} \otimes S^2\sQ |_{L_i} (-1)) =\bC$.
 \item $H^5(\bigwedge^5S^2\sQ^\vee |_{L_i} \otimes S^2\sQ  |_{L_i}(-2)) =\bC$.
\end{enumerate}
\end{lemma}

\begin{proof}
 The bundle $\bigwedge^kS^2\sQ^\vee |_{L_i} \otimes S^2\sQ |_{L_i} (m) $ is the direct sum of homogeneous vector bundles with the following weights (with multiplicity one).
 The assertion follows from the Borel-Weil-Bott theorem.
 \begin{longtable}{llll}
 \toprule
$(m,k)$ & $w$ & $w+\rho$ & index \\ \midrule
$(-1,0)$&$(0,0,1)$ & $(5,3,3)/2$ & $\blank$\\
$(-1,1)$&$(-1,-1,2)$ & $(3,1,5)/2$ & $2$\\
&$(-1,-1,1)$ & $(3,1,3)/2$ & $\blank$\\
&$(-1,-1,0)$ & $(3,1,1)/2$ & $\blank$\\
$(-1,2)$&$(-2,-2,2)$ & $(1,-1,5)/2$ & $\blank$\\
&$(-2,-2,1)$ & $(1,-1,3)/2$ & $\blank$\\
&$(-2,-2,0)$ & $(1,-1,1)/2$ & $\blank$\\
$(-1,3)$&$(-3,-3,1)$ & $(-1,-3,3)/2$ & $\blank$\\ \midrule
$(-2,0)$&$(-1,-1,1)$ & $(3,1,3)/2$ & $\blank$\\
$(-2,1)$&$(-2,-2,2)$ & $(1,-1,5)/2$ & $\blank$\\
&$(-2,-2,1)$ & $(1,-1,3)/2$ & $\blank$\\
&$(-2,-2,0)$ & $(1,-1,1)/2$ & $\blank$\\
$(-2,2)$&$(-3,-3,2)$ & $(-1,-3,5)/2$ & $5$\\
&$(-3,-3,1)$ & $(-1,-3,3)/2$ & $\blank$\\
&$(-3,-3,0)$ & $(-1,-3,1)/2$ & $\blank$\\
$(-2,3)$&$(-4,-4,1)$ & $(-3,-5,3)/2$ & $\blank$\\
\bottomrule
\end{longtable}
\end{proof}

\section{Proofs of main theorems}
\label{section:proof}

Let $(S,h)$ be a \emph{general}{ primitively polarized K3 surface of genus $13$.
$T = \sM_S(2,h,6)$ is the moduli space of stable vector bundles with Mukai vector $v = (2,h,6)$.
In the Mukai lattice, the vector $(-1,0,3)$ is in $v ^\perp$.
Thus the class $\hat h = [(-1,0,3)] \in v^\perp/\bZ v$ defines a polarization of degree $6$ \cite{Muk87} (cf.\ \cite[Theorem 6.1.14]{HL10}).

By Theorem~\ref{theorem:CI_K3_in_Gr(8,2)} and Theorem~\ref{theorem:cohomology_S}~\ref{theorem:cohomology_S_6}, the surface $S' =\Gr(\bC^8,2,\sP) \cap (s_1)_0 \cap  (s_2)_0$ for general $s_i$ has the primitive polarization $h'=\alpha|_{S'}$, and $\sQ|_S$ is simple.
Thus, by \cite[Proposition 4.1]{Muk92}, the general pair $(S,\sE)$ ($\sE \in T$) degenerates to $(S',\sQ|_S)$, and hence $(S,\sE)$ enjoys many properties as $(S',\sQ|_S)$ does.
Theorem~\ref{theorem:cohomology_S} and Corollary~\ref{corollary:stability} implies:
\begin{corollary}\label{corollary:properties_of_E}
If $\sE \in T$ is general, then the following hold:
\begin{enumerate}
 \item $H^0(\sE) =\bC^8$ and $H^i(\sE) =0$ for $i>0$.
 \item $\sE$ is globally generated.
 Thus we have a map $S \to \Gr(\bC^8,2)$.
 \item The natural map
 \[
S^2H^0(\sE) \to H^0(S^2\sE)
\]
is surjective with two dimensional kernel $\sP$.
This induces a map $S \to \Gr(\bC^8,2,\sP)$.
\item $\sP$ is a simple pencil of quadric forms.
 \item The natural map
\[
\bigwedge^2 H^0(\sE) \to H^0(\bigwedge^2\sE)
\]
is a surjective map.

 \item The induced map $S \to \Gr(\bC^8,2,\sP)$ is an embedding.
 \item $H^0(\sU_z|_S) =\bC^8$ and $H^i(\sU_z|_S) =0$ for $i>0$ and $z \in C$.
 \item For a general point $p \in C$, the map $H^0(\sU_p) \to H^0(\sU_p|_S)$ is an isomorphism.
 \item $H^0(\sU_z|_S(-h)) =0$. In particular $\sU_z|_S$ is stable for any $z \in C$.
\end{enumerate}
\end{corollary}

\begin{proof}[Proof of Theorem~\ref{theorem:general_K3_in_G(8,2)}]
We first construct the inverse map $\Psi^{-1}$.
Let $S$ be a general primitively polarized K3 surface of genus $13$, and $\sE \in T$ be a general stable vector bundle with Mukai vector $(2,h,6)$.
Then $\sE$ embeds $S$ into $\Gr(\bC^8,2,\sP) =\sN^{-}$ for the unique simple pencil of quadric forms $\sP \subset S^2H^0(\sE)$.
Moreover $\sU_z|_S$ ($z \in C$) are vector bundles of rank two with $H^0(\sU_z|_S) =\bC^8$ and $H^i(\sU_z|_S) = 0$ for $i>0$.
By Theorem~\ref{proposition:restricted_VB_on_C}, $(\pr_2)_*(\sU|_{S \times C})$ is a vector bundle of rank $8$ with degree $4d-6$.
Thus, by Theorem~\ref{proposition:VB_on_C}, the difference of the degree of the natural restriction map
\[
\Ev \colon (\pr_2)_*\sU \to (\pr_2)_*(\sU|_{S \times C})
\]
is two.
Since the map $H^0(\sU_p) \to H^0(\sU_p|_S)$ is an isomorphism for a general point $p \in C$, $\Ev$ is generically an isomorphism.
Hence we find exactly two points $p_1$ and $p_2 \in C$ such that $H^0(\sU_{p_i}) \to H^0(\sU_{p_i}|_S)$ are not an isomorphism.
At these two points, each map has one dimensional kernel spanned by some $s_i \in H^0(\sU_{p_i}) $.
(Alternatively, we can find two sections $s_i$ by using Theorem~\ref{theorem:cohomology_S}~\ref{theorem:cohomology_S_10}).
This defines the inverse map $\Psi^{-1}$.
Note that, by specializing to the complete intersection case, we see that $S = \Gr(\bC^8,2,\sP) \cap (s_1)_0 \cap (s_2)_0$.

Next, we construct the moduli map $\Psi$ in the theorem.
From the above step, we see that, for general points $p_1$ and $p_2 \in C$ and for general sections $s_i \in H^0(\sU_{p_i})$, the complete intersection variety $S=\Gr(\bC^8,2,\sP) \cap (s_1)_0 \cap (s_2)_0$ is a \emph{general} polarized K3 surface of genus $13$ and that $\sQ|_S$ is stable.
This defines the moduli map $\Psi$.
By the constructions, these maps $\Psi$ and $\Psi^{-1}$ are inverse to each other.
We complete the proof.
\end{proof}

We proceed to prove Theorem~\ref{theorem:general_K3_in_N-}.
Let $C$ be a general curve of genus $3$, or a general \emph{hyperelliptic} curve of genus $3$.
Let  $p_1$, $p_2\in C$ be general two points on $C$ and $s_i \in H^0(\sU_{p_i})$ general sections.
Set $S = (s_1)_0 \cap (s_2)_0 \subset \sN^-$.
If $C$ is not hyperelliptic, then, by specializing to the hyperelliptic case, we see that $S$ is a general K3 surface and that $\sU_z|_S$ is stable.
Thus we have a map $\varphi \colon C \to T  \subset \bP^4$ to a sextic K3 surface $T$.

\begin{lemma}\label{lemma:degree_of_C}
The degree of $\varphi^*\hat h$ is $6$.
In particular, the image $C' \coloneqq \varphi(C) \subset T \subset  \bP^4$ is a singular hyperplane section of $T$.
\end{lemma}

\begin{proof}
The degree of $\varphi^*\hat h$ is as follows (see, e.g., \cite[Theorem~8.1.5]{HL10}:
\[
\deg((\pr_2)_*(\pr_1^*(T^{-1}(-1,0,3)) \cdot \ch(\sU|_{S \times C}) \cdot \pr_1^*(\td(S)) )).
\]
By  a similar calculation as in the proof of Proposition~\ref{proposition:restricted_VB_on_C}, we have 
\[
\deg((\pr_2)_*(\pr_1^*(T^{-1}(-1,0,3)) \cdot \ch(\sU|_{S \times C}) \cdot \pr_1^*(\td(S)) )) = 6.
\]
Thus the image $\varphi(C) \subset T \subset  \bP^4$ is a hyperplane section of $T$.
Since the genus of $C$ is three, we see that the image is a singular hyperplane section of $T$.
\end{proof}

Let $C'  \subset T \subset  \bP^4$ be a singular hyperplane section of $T$, and $C \to C'$ the normalization.
On $S \times T$, there exists the universal $\bP^1$-fibration $\bP$.
The restriction of this $\bP^1$-fibration to $S \times C$ is a projectivization of a vector bundle, since the Brauer element of $\bP$ is from $T$ only.
Fix such a vector bundle $\sF$.
If $C'$ is the singular hyperplane section as in Lemma~\ref{lemma:degree_of_C}, then the bundle $\sF|_{\{x\}\times C}$ is the stable bundle $\sU|_{\{x\}\times C}$.
By  considering the specialization to this case, we have:
\begin{lemma}\label{lemma:nodal_section}
Let $[C'] \in T^\vee$ is a general singular hyperplane section of $T$, then the bundle $\sF|_{\{x\}\times C}$ is a stable bundle of odd degree on the normalization $C$ of $C'$.
\end{lemma}

\begin{proof}[Proof of Theorem~\ref{theorem:general_K3_in_N-}]
Let $C$ be a general curve of genus $3$,  $p_1$, $p_2\in C$ general two points on $C$ and $s_i \in H^0(\sU_{p_i})$ general sections.
Set $S = (s_1)_0 \cap (s_2)_0 \subset \sN^-$.
Then, $S$ is a \emph{general} primitively polarized K3 surface of genus $13$ and, by Lemma~\ref{lemma:degree_of_C}, the induced map $\varphi \colon C \to T$ maps $C$ to a singular hyperplane section $[C'] \in T^\vee$ of the sextic K3 surface $T \subset \bP^4$.
This defines the moduli map $\Phi$.

We construct the inverse map.
Let $(S,h) \in \sF_{13}$ be a general primitively polarized K3 surface of genus $13$ and $C' \subset T$ be a general singular hyperplane section of $T \subset \bP^4$.
Then, by Lemma~\ref{lemma:nodal_section}, $\sF|_{\{x\}\times C}$ is stable of odd degree on the normalization $C$ of $C'$.
This defines a map $S \to \sN^-$.
By specializing to the case where the curve $C$ is a hyperelliptic curve, we  see that this map specializes to the inclusion $S' = (s_1)_0 \cap (s_2)_0 \subset \Gr(\bC^8,2,\sP)$ for a unique simple pencil of quadrics $\sP \subset S^2\bC^8$, general two points $p_1$, $p_2 \in C$ and general sections $s_i \subset H^0(\sU_{p_i})$.
By a similar argument to the proof of Theorem~\ref{theorem:general_K3_in_G(8,2)} or considering the specialization to the case where the curve $C$ is a hyperelliptic curve,  we also find two (general) points $p_1$, $p_2 \in C$ and (general) sections $s_i \subset H^0(\sU_{p_i})$ such that $S = (s_1)_0\cap (s_2)_0 \subset \sN^-$.
This defines the inverse map $\Phi^{-1}$.
By the  constructions, these maps are inverse to each other.
\end{proof}

\section{the unequal parity case and parabolic analogues}
\label{section:unequal_parity}

In this section, we describe the K3 surface $S$ in Example~\ref{ex:complete_intersection}~\ref{ex:complete_intersection2}.
We first relate $S$ to a subvariety of the $G_2$-homogeneous variety $K(G_2)$ \cite[Remark 1]{Muk06}, and then to a subvariety of the moduli space $\sN_{0,7}$ of parabolic $2$-bundles over a $7$-pointed rational curve  $(\bP^1; x_1, \ldots, x_7)$.
By using the birational geometry of $\sN_{0,7}$  \cite{Bau91}, \cite[Example~12.57]{Muk03}, \cite{AC17}, we describe rational curves on $S$.

\subsection{Unequal parity case, $K(G_2)$ and $\sN_{0,7}$}
Let $\sP = \langle q_1, q_2\rangle \subset S^2\bC^8$ be a general pencil of quadric forms, and $C$ the associated hyperelliptic curve of genus $3$.
Take a general point $p_1$ and set $p_2 \coloneqq \iota(p_1)$.
We also take general sections $s_i \in H^0(\sU_{p_i})$.
Then $S\coloneqq  (s_1)_0\cap(s_2)_0$ is a K3 surface of genus $13$.

The image $\bP^1$ of the double cover $f \colon C \to \bP^1$ is identified with the pencil of quadric forms $\sP$.
Let $q \coloneqq f(p_i)$ be the corresponding quadric form.
Recall that $\sN^-$ is the intersection of orthogonal Grassmann varieties $\Gr(\bC^8,2,q')$ for $q' \in \sP$.
Note also that $\sU_{p_i}$  are the restrictions of the spinor bundles $\sS_\pm$ on $\Gr(\bC^8,2,q)$.
By interchanging the indices if necessary, we may assume that $\sS_+|_{\sN^-} = \sU_{p_1}$ and $\sS_-|_{\sN^-} = \sU_{p_2}$.

Let $\overline{s_i} \in \sS_{\pm}$ be the spinors such that $\overline{s_i}|_{\sN^-} = s_i$.
Then the K3 surface $S = (s_1)_0\cap(s_2)_0$ is the complete intersection variety
\[
 (\overline{s_1})_0\cap(\overline{s_2})_0 \cap (q')_0
\]
in $\Gr(\bC^8,2,q)$ with respect to the homogeneous vector bundle
\[
\sS_{+} \oplus\sS_{-} \oplus S^2\sQ.
\]

Consider first the zero locus
\[
 (\overline{s_1})_0\cap(\overline{s_2})_0.
\]
Recall that, by the triality of $D_4$, $\Gr(\bC^8,2,q)$ is isomorphic to $\Gr(H^0(\sS_+),2,q_+)$ for a quadric form $q_+ \in S^2H^0(\sS_+)$.
Via this isomorphism, the vector bundles $\sQ$ and $\sS_-$ are the spinor bundles for $\Gr(H^0(\sS_+),2,q_+)$.
Thus the variety $(\overline{s_1})_0\cap(\overline{s_2})_0$ is isomorphic to the $G_2$ homogeneous variety $K(G_2)$ and the restrictions of $\sS_\pm$ and $\sQ$ are isomorphic to the rank $2$ tautological vector bundle $\sE$ on $K(G_2)$.
From the view point of the Dynkin diagrams, this corresponds to the folding of $D_4$ to $G_2$:
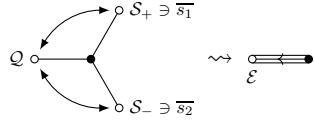
\[
\dynkin[%
edge length=.75cm, labels={\sQ,,\sS_+\ni \overline{s_1},\sS_-\ni \overline{s_2}}, involutions={14;31},label directions={left,,,}
]{D}{o*oo}
\rightsquigarrow
\dynkin[%
edge length=.75cm,backwards,labels={,\sE}]{G}{*o}
\]
Since $\sP$ and, hence, $q'$ is general, the restriction $q'|_{K(G_2)} \in S^2\sE$ is a general section.
Hence our K3 surface $S$ is isomorphic to the zero locus of a general section of $S^2\sE$ on $K(G_2)$, which is the K3 surface in \cite[Remark~1]{Muk06}.

Next, we describe this K3 surface in the moduli of parabolic vector bundles on $7$ pointed projective line.
Let $\widetilde {q_1}$, $\widetilde {q_2} \in S^2\bC^7$ be general elements and set $\widetilde\sP \coloneqq \langle \widetilde{q_1}, \widetilde {q_2} \rangle$.
Similarly to the case of hyperelliptic curves, we can associates $\widetilde \sP$ to a $7$ pointed projective line $(\bP^1;x_1,\ldots,x_7)$.
Then the complete intersection subvariety
\[
\Gr(\bC^7,2,\widetilde \sP) \coloneqq \Gr(\bC^7,2,\widetilde {q_1})  \cap \Gr(\bC^7,2,\widetilde {q_2}) \subset \Gr(\bC^7,2)  
\]
is the moduli $\sN_{0,7}$ of stable parabolic vector bundles on $(\bP^1;x_1,\ldots,x_7)$ of weight $(\frac{1}{2}, \ldots, \frac{1}{2})$ \cite{Cas15}.
Let $\widetilde{\sS}$ be the spinor bundle on the orthogonal Grassmann variety $\Gr(\bC^7,2,\widetilde{q_1})$.
Then the restriction $\widetilde\sS|_{\sN_{0,m}}$ is isomorphic to $\sE_x$, where $\sE$ is the normalized universal bundle on $\sN_{0,7} \times \bP^1$ and $x \in \bP^1$ is the point corresponding to $q_1$.

Take a general section $\widetilde s \in H^0(\widetilde\sS)$ and consider the zero locus
\[
S= (\widetilde s |_{\sN_{0,7}})_0.
\]
By the definition, this variety is isomorphic to the complete intersection variety
\[
(\widetilde s )_0 \cap (\widetilde{q_2})_0 \subset \Gr(\bC^7,2,\widetilde{q_1})
\]
with respect to the vector bundle $\widetilde\sS \oplus S^2\sQ$.
The zero locus $ (\widetilde{s})_0 \subset \Gr(\bC^7,2,\widetilde{q_1})$ is isomorphic to $K(G_2)$, and the restrictions of $\sQ$ and $\widetilde\sS$ are isomorphic to the tautological $2$-bundle $\sE$:
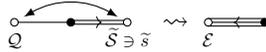
\[
\dynkin[%
edge length=.75cm, labels={\sQ,,\widetilde{\sS} \ni \widetilde{s}}, involutions={31},label directions={,,}
]{B}{o*o}
\rightsquigarrow
\dynkin[%
edge length=.75cm,backwards,labels={,\sE}]{G}{*o}
\]
Thus $S$ is isomorphic to the zero locus $(\widetilde{q_2}|_{K(G_2)})_0 \subset K(G_2)$ of a section of $S^2\sE$.
By summarizing the above two cases, we have:
\begin{proposition}
For general choices of sections,  the following vector bundles define the same K3 surface as complete intersection varieties:
\begin{enumerate}
 \item $\sU_{p}\oplus \sU_{\iota(p)}$ on $\sN^-$ for a general hyperelliptic curve $C$ and a general point $p \in C$.
 \item $S^2\sE$ on $K(G_2)$.
 \item $\sE_x$ on $\sN_{0,7}$ for a general point $x \in (\bP^1;x_1,\ldots,x_7)$.
\end{enumerate}
 
\end{proposition}

\begin{remark}
The parabolic analogue of Calabi-Yau manifolds studied in Theorem~\ref{theorem:CI_CY_in_N-} works for all the moduli space $\sN_{0, m}$ of parabolic 2-bundles over an $m$-pointed rational curve  $(\bP^1; x_1, \ldots, x_m)$ with odd $m$.
As a complete intersection with respect to the rank 2 bundle  $\sE_x$, $x \in \bP^1$, we obtain a smooth scheme of dimension $m-5$ with trivial canonical bundle, where  $\sE_x$  is the restriction of the normalized universal bundle $\sE$ to  $\sN_{0, m} \times {x}$.
\end{remark}

\subsection{Geometry of $\sN_{0,7}$ and rational curves on $S$}
For the most democratic weight $(\frac{1}{2}, \ldots, \frac{1}{2})$, the moduli space $\sN_{0,7}$  is the (complete) intersection of pairs of $\SO(7)$-varieties in  $\Gr(\bC^7, 2)$, and our K3 surface $S$ is a complete intersection variety with respect to the rank $2$ bundle $\sE_x$.
Moving weights from $(\frac{1}{2}, \ldots, \frac{1}{2})$ to unbalanced ones, we obtain several birational models of the moduli space  $\sN_{0,7}$ \cite{Bau91}, \cite[Example~12.57]{Muk03}, \cite{AC17}.
The typical one is the blow-up of  $\bP^4$ at $7$ points in general position.
By using the explicit description of the birational maps between these two models, we describe rational curves on $S$.

By the definition, $\Gr(\bC^7,2,\widetilde \sP)$ parametrizes the lines contained in the complete intersection $Z$ of two quadrics in $\bP^6$.
Without loss of generality, we may assume that $\widetilde \sP$ and, hence, $Z$ is defined by the following quadric forms:
\begin{align*}
&x_1^2 + \cdots + x_7^2, \\
&a_1 x_1^2 + \cdots + a_7 x_7^2.
\end{align*}

Let $\sigma_i$ be the involution of $Z$ induced by $x_i \mapsto -x_i$ and set $\sigma_I \coloneqq \prod_{i \in I} \sigma_i$ for $I \subset \{1,\dots,7\}$.
The group $G$ generated by these involutions is isomorphic to $(\bZ/2\bZ)^6$.
It is known that there are $64$ planes on $Z$ and the group $G$ acts on the set of planes freely and transitively.

To a plane $M$, we can attach a prime effective divisor $E_M \subset \Gr(\bC^7,2,\widetilde\sP)$.
In the following, we fix a plane $M$ and denote simply by $E_I$ the divisor $E_{\sigma_I(M)}$.
Note that $E_I = E_{\{1,\dots,7\} \setminus I}$ and, hence, $E_I$ with odd $|I|$ define $64$ prime divisors.
The results of \cite{Bau91}, \cite[Example~12.57]{Muk03}, \cite{AC17} are summarized as follows:
\begin{theorem}
 Each plane $M$ on $Z$ defines a birational map
 \[
 \rho_M \colon \Gr(\bC^7,2,\widetilde\sP) \dashrightarrow  B
 \]
 to the blow up of $\bP^4$ along $7$ points in general position such that:
\begin{enumerate}
 \item The image of $E_M$ is the strict transform of the secant variety of  the rational normal quartic curve passing through the $7$ points.
 \item $E_i = E_{\sigma_i(M)}$ ($i=1$, \dots, $7$) is mapped to one of the exceptional divisors over the $7$ points.
 \item $\rho_M^{-1}$ is a composite of (anti-)flips. It transforms the following $\bP^1$ to $\bP^2$:
\begin{itemize}
 \item the strict transforms of the lines passing through two of the seven points.
 \item the strict transform of rational normal quartic curve passing through the $7$ points.
\end{itemize}
\end{enumerate}
\end{theorem}

Let $H_M \in \Pic(\Gr(\bC^7,2,\widetilde\sP))$ be the class corresponding to the pullback of the hyperplane in $\bP^4$.
Then we have the following relations among $E_I$, $H_M$ and $-K_{\sN_{0,7}}$ \cite{AC17}:
\begin{equation}\label{eq:relations}
\begin{split}
& E_M =E_{\emptyset} =  \frac{1}{5}\left( -3K_{\sN_{0,7}} -\sum_{i \in \{1,\dots,7\}}E_i\right),\\
& E_I =
 \begin{cases}
 \frac{1}{5}\left( -K_{\sN_{0,7}} -2\sum_{i \not \in I}E_i + 3\sum_{i  \in I}E_i\right)  & (|I|=3),\\
  \frac{1}{5}\left( 2 (-K_{\sN_{0,7}}) -4\sum_{i \not  \in I}E_i + \sum_{i \in I}E_i\right)  & (|I|=5),
 \end{cases}\\
&H_M = \frac{1}{5}\left( -K_{\sN_{0,7}} +3\sum_{i \in \{1,\dots,7\}}E_i \right).
\end{split}
\end{equation}

\begin{theorem}
Set $h \coloneqq -K_{\sN_{0,7}} |_S$ and $e_I \coloneqq E_I |_S$.
Then the following hold:
\begin{enumerate}
 \item Each $e_I$ is a sextic rational curve on $(S,h)$.
 Intersection numbers between these rational curves are as follows:
\begin{align*}
 &e_{\emptyset} \cdot e_i =4,\\
 &e_{\emptyset} \cdot e_{I} =
\begin{cases}
 2 & (|I|=3),\\
 0  &(|I|=5).
\end{cases}
\end{align*}
 \item $H_M|_S = \frac{1}{5}\left( h +3\sum_{i \in \{1,\dots,7\}}e_i \right)$ defines a base point free linear system on $S$.
 The corresponding map contracts the $7$ rational curves $e_i$, and sends $S$ to the complete intersection  of a quadric and a cubic hypersurfaces, which is singular at  $7$ points.
\end{enumerate}
\end{theorem}

\begin{proof}
Let $W$ be one of the flipping center of $\rho_M$.
The restriction of $ \widetilde \sS$ to this plane $W$ is a direct sum of line bundles $\sO \oplus \sO(1)$.
Thus $S \cap W =\emptyset$.
In particular, $e_i \cap e_j = E_i \cap E_j \cap S= \emptyset$ for $i \neq j$, since $E_i \cap E_j$ is a union of some flipping centers of $\rho_M$.

Since $S$ does not meet the indeterminacy of $\rho_M$, $H_M|_S$ is a base point free divisor on $S$.
Moreover the induced map $S \to \bP^4$ is a birational map whose exceptional locus is the union of $E_i\cap S$.
Hence, $e_i$ are (disjoint) rational curves.
By using relations~\eqref{eq:relations}, we see that $H_M|_S$ is of degree $6$.
Since the image of $S$ in $\bP^4$ passes through the $7$ points in general position, it is non-degenerate and, thus, the map $S \to \bP^4$ is defined by the complete linear system of $H_M|_S$.
Therefore the image in $\bP^4$ is the complete intersection of a quadric and a cubic hypersurfaces (\cite{Sai74}).

Since $\rho_M|_S$ contracts $e_i$, we have $H_M|_S\cdot e_i =0$.
By relations~\eqref{eq:relations}, we have
\begin{align*}
& h \cdot e_i =6,\\
&e_{\emptyset}\cdot e_i = 4,\\
&e_{\emptyset}  \cdot e_I =
\begin{cases}
  2& (|I|=3),\\
  0  & (|I|=5).
\end{cases}
\end{align*}

Finally, since each $E_I$ is the translation of $E_M$, we see that $e_I$ are all sextic rational curves.
\end{proof}

\bibliographystyle{amsalpha}
\bibliography{references}

\end{document}